\documentclass[12pt,reqno]{amsart}
\usepackage{amsmath}
\usepackage{amssymb}
\usepackage{amstext}
\usepackage{a4wide}
\usepackage{graphicx}
\allowdisplaybreaks \numberwithin{equation}{section}
\usepackage{color}
\usepackage{cases}

\numberwithin{equation}{section}

\newtheorem{theorem}{Theorem}[section]

\newtheorem{corollary}[theorem]{Corollary}
\newtheorem{lemma}[theorem]{Lemma}

\theoremstyle{definition}

\newtheorem{definition}[theorem]{Definition}

\theoremstyle{remark}
\newtheorem{remark}[theorem]{Remark}

\begin{document}

\title
{On the steady axisymmetric  vortex rings for 3-D incompressible Euler flows}

 \author{Daomin Cao, Weicheng Zhan}

\address{Institute of Applied Mathematics, Chinese Academy of Sciences, Beijing 100190, and University of Chinese Academy of Sciences, Beijing 100049,  P.R. China}
\email{dmcao@amt.ac.cn}

\address{Institute of Applied Mathematics, Chinese Academy of Sciences, Beijing 100190, and University of Chinese Academy of Sciences, Beijing 100049,  P.R. China}
\email{zhanweicheng16@mails.ucas.ac.cn}


\begin{abstract}
In this paper, we study nonlinear desingularization of steady vortex rings of three-dimensional incompressible Euler flows. We construct a family of steady vortex rings (with and without swirl) which constitutes a desingularization of the classical circular vortex filament in $\mathbb{R}^3$. The construction is based on a study of solutions to the similinear elliptic problem
\begin{equation*}
-\frac{1}{r}\frac{\partial}{\partial r}\Big(\frac{1}{r}\frac{\partial\psi^\varepsilon}{\partial r}\Big)-\frac{1}{r^2}\frac{\partial^2\psi^\varepsilon}{\partial z^2}=\frac{1}{\varepsilon^2}\left(g(\psi^\varepsilon)+\frac{f(\psi^\varepsilon)}{r^2}\right),
\end{equation*}
where $f$ and $g$ are two given functions of the Stokes stream function $\psi^\varepsilon$, and $\varepsilon>0$ is a small parameter.
\bigskip

\end{abstract}

\maketitle

\section{Introduction and Main results}

In this paper, we study the three-dimensional (3-D) axisymmetric Euler flows. The motion of general incompressible steady Euler fluid in $\mathbb{R}^3$ is described by the following equations

\begin{equation}\label{1-1}
~(\mathbf{v}\cdot\nabla)\mathbf{v}=-\nabla P,
\end{equation}
\begin{equation}\label{1-2}
 \nabla\cdot\mathbf{v}=0,
\end{equation}
where $\mathbf{v}=[v_1,v_2,v_3]$ is the velocity field and $P$ is the scalar pressure. Let $\{\mathbf{e}_r, \mathbf{e}_\theta, \mathbf{e}_z\}$ be the usual cylindrical coordinate frame. Then if the velocity field $\mathbf{v}$ is axisymmetric, i.e., $\mathbf{v}$ does not depend on the $\theta$ coordinate, it can be expressed in the following way
\begin{equation*}
  \mathbf{v}=v^r(r,z)\mathbf{e}_r+v^\theta(r,z)\mathbf{e}_\theta+v^z(r,z)\mathbf{e}_z.
\end{equation*}
 The component $v^\theta$ in the $\mathbf{e}_\theta$ direction is called the swirl velocity. Let $\pmb{\omega}:=\nabla\times\mathbf{v}$ be the corresponding vorticity field. Taking into account $(\mathbf{v}\cdot\nabla)\mathbf{v}=\nabla(|\mathbf{v}|^2/2)-\mathbf{v}\times \pmb{\omega}$, we rewrite the momentum equation \eqref{1-1} as follows:
 \begin{equation}\label{1-3}
   \mathbf{v}\times \pmb{\omega}=\nabla\left(P+\frac{|\mathbf{v}|^2}{2} \right).
 \end{equation}
 In cylindrical coordinates $(r,\theta,z)$, the vorticity field $\pmb{\omega}=(\omega^r,\omega^{\theta},\omega^z)$ is given by
\begin{equation}\label{1-4}
  \pmb{\omega}=-\partial_z v^{\theta}\mathbf{e}_r+\left(\partial_z v^r-\partial_r v^z\right)\mathbf{e}_\theta+\frac{1}{r}\partial_r (r v^{\theta})\mathbf{e}_z.
\end{equation}
 By the continuity equation \eqref{1-2}, we can find a Stokes stream function $\psi(r,z)$ such that
\begin{equation}\label{1-5}
   v^r=-\frac{1}{r}{\partial_z\psi},\ \  \  v^z=\frac{1}{r}{\partial_r\psi}.
\end{equation}
 Let $\zeta:={\omega^{\theta}}/{r}$ and $\xi:=r v^{\theta}$. By equation \eqref{1-1}, we see that
 \begin{equation}\label{1-6}
\mathbf{v}\cdot\nabla\xi=0,
\end{equation}
 where $\mathbf{v}\cdot\nabla=v^r\partial_r+v^z\partial_z$. The combination of \eqref{1-5} and \eqref{1-6} then leads to
 \begin{equation*}
   \nabla^\perp \psi\cdot \nabla\xi=0,\ \ \text{with}\ \ \nabla^\perp:=(-\partial_z, \partial_r),
 \end{equation*}
 which holds automatically if there exists some function $H:\mathbb{R}\to \mathbb{R}$ satisfying
\begin{equation}\label{1-7}
\xi=H(\psi).
\end{equation}
Substituting the expressions $\eqref{1-6}$ and $\eqref{1-7}$ into \eqref{1-4}, we obtain
\begin{equation}\label{1-8}
  \pmb{\omega}=-\frac{1}{r}H'(\psi)\partial_z\psi\mathbf{e}_r+r\mathcal{L}\psi\mathbf{e}_\theta+\frac{1}{r}H'(\psi)\partial_r\psi\mathbf{e}_z,
\end{equation}
where
\begin{equation*}
  \mathcal{L}\psi:=-\frac{1}{r}\frac{\partial}{\partial r}\Big(\frac{1}{r}\frac{\partial\psi}{\partial r}\Big)-\frac{1}{r^2}\frac{\partial^2\psi}{\partial z^2}.
\end{equation*}
Substituting \eqref{1-8} into the equation \eqref{1-3}, after simple transformations, we obtain
\begin{equation*}
  \mathcal{L}\psi=-\frac{\partial_r B}{\partial_r \psi}+\frac{H(\psi)H'(\psi)}{r^2}=\zeta,
\end{equation*}
where $B:=P+|\mathbf{v}|^2/2$ is the Bernoulli integral (or dynamic pressure). By Bernoulli's theorem, $B$ can only be a function of $\psi$ alone. Thus we are led to the Bragg-Hawthorne equation \cite{BH, Fre07}
\begin{equation}\label{1-10}
\mathcal{L}\psi=-B'(\psi)+\frac{1}{r^2}H(\psi)H'(\psi).
\end{equation}
This family of equations is also sometimes referred to as the Long-Squire equation (see \cite{Long53, Sq56}). It is usually known as the Grad-Shafranov equation in plasma physics (see \cite{An98}). Once the Stokes stream function $\psi$ satisfying \eqref{1-10} for an arbitrary choice of generators $(B, H)$ is obtained, one can easily construct the corresponding solutions of the original or primitive variables $(\mathbf{v}, P)$ and vice versa.

In this paper, we shall refer to the axisymmetric flows  as ``vortex rings" if there is a toroidal region inside of which the $\theta$-component of vorticity $\omega^\theta$ is nonzero (the core), and outside of which the flow is irrotational. Vortex rings are an intriguing marvel of fluid dynamics that are ubiquitous throughout nature. This type of solutions of the Euler equations are classical objects of study in fluid dynamics. The study of vortex rings can be traced back to the works of Helmholtz \cite{He} in 1858 and Kelvin \cite{Tho} in 1867. For a detailed and historical description of this problem, we refer to \cite{Cao2, DV, BF1}. In addition, \cite{Akh, Lim95, MGT, Sa92} are some good historical reviews of the achievements in experimental, analytical, and numerical studies of vortex rings.

Generally, vortex rings are divided into two cases according to whether the swirl velocity $v^\theta$ (or, equivalently, $H$) is zero. One is the swirling case, and the other is the non-swirling case. Global existence of vortex rings without swirl (i.e., $\mathbf{v}^\theta\equiv0$) was first established by Fraenkel and Berger \cite{BF1}. After this pioneering work, much work is devoted to further studying the existence of vortex rings without swirl. See, e.g., \cite{AS,BB, B1, Cao2, DV, BF1, FT, Ni, No,Yang2} and the references therein. Compared with vortex rings without swirl, the dynamics of vortex rings with swirl seem to have received relatively less attention. Limited work has been done in this aspect. The only explicit solution for vortices with swirl was first found by Hicks \cite{H} and rediscovered by Moffatt \cite{M, M2} (see also \cite{Abe, Fre07}). Fraenkel \cite{Fra} and Tadie \cite{Ta2} generalize the variational approach in \cite{BF1} for vortex rings without swirl. In \cite{M2}, Moffatt adopted magnetic relaxation methods to obtain a wide family of vortex rings with swirl. However, the functions $B$ and $H$ in \cite{M2} are not prescribed, but determined a posteriori. Turkington \cite{Tur89} proposed a variational principle for the potential vorticity $\zeta$ and obtained existence of solutions of \eqref{1-10} for some special prescribed $B$ and $H$. In the case of bound domains, instead of using a variational approach, Elcrat and Miller \cite{EM} developed an iterative procedure to solve $\eqref{1-10}$ for some prescribed $B$ and $H$. Very recently, Abe \cite{Abe} studied the existence of vortex rings in Beltrami flows, which corresponds to constant dynamic pressure $B$.

The purpose of this paper is to construct a family of steady vortex rings (with and without swirl) which constitutes a desingularization of the classical circular vortex filament in $\mathbb{R}^3$. From the above discussion, we are thus interested in studying the asymptotics of solutions of
\begin{equation}\label{aim}
  \mathcal{L}\psi^\varepsilon=\frac{1}{\varepsilon^2}\left(g(\psi^\varepsilon)+\frac{f(\psi^\varepsilon)}{r^2}\right)
\end{equation}
where $f$, $g$ are two given functions, and $\varepsilon>0$ is a parameter.
The Stokes stream function $\psi^\varepsilon$ will bifurcate from Green's funtion for $\mathcal{L}$ as $\varepsilon\to 0^+$. This kind of bifurcation phenomenon is called ``nonlinear desingularization''(refer to \cite{BF2}).

For the non-swirling case ($f\equiv 0$), several desingularization results have been obtained. Nonlinear desingularization for general free-boundary problems was studied in \cite{BF2}, but asymptotic behavior of the solutions they constructed could not be studied precisely because of the presence of a Lagrange multiplier in the nonlinearity $g$. In \cite{FT}, Friedman and Turkington obtained desingularization results of vortex rings without swirl in the whole space when the vorticity function $g$ is a step function. In \cite{Ta}, Tadie studied the asymptotic behavior by letting the flux diverge. Yang studied the asymptotic behaviour of a family of solutions $\psi^\varepsilon$ of $\eqref{aim}$ for some smooth function $g$ \cite{Yang2}. However, their limiting objects are degenerate vortex rings with vanishing circulation. In \cite{DV}, de Valeriola and Van Schaftingen considered this problem for $g(t)=t_+^p $ with $p>1$, where $t_+:=\max\{0,t\}$. Recently, Cao et al. \cite{Cao2} further investigated desingularization of vortex rings without swirl when $g$ is a step function and generalized the results in \cite{FT} to some extent. In \cite{Cao4}, Cao et al. studied the asymptotic behavior of vortex rings with some general vorticity functions.

For the swirling case, it seems that there are very few results in this direction. In \cite{Ta2}, Tadie considered this problem outside infinite cylinders and investigated the asymptotic behaviour by letting the flux diverge. In \cite{Tur89}, Turkington constructed a two-parameter (w.r.t. $\varepsilon>0$ and $\alpha\ge 0$) family of desingularized steady solutions which corresponds to \eqref{aim} with
\begin{equation*}
  f(t)=t_+, \ \ g^{\varepsilon,\alpha}(t)=\varepsilon \alpha \chi_{_{\{t>0\}}},
\end{equation*}
where $\chi_{_A}$ denotes the characteristic function of $A$.

In this paper, we further study this problem.  For technical reasons, we make the following assumptions on $(f,g)$.
\begin{itemize}
  \item [$(a1)$] $f\in C(\mathbb{R})$ and $g\in C(\mathbb{R}\backslash \{0\})$ are both nonnegative and nondecreasing.
\end{itemize}
Set
\begin{equation*}
  i(r,t)=g(t)+\frac{f(t)}{r^2}
\end{equation*}
and $I(r,t)=\int_{0}^{t}i(r,s)ds$. Let $J(r,\cdot)$, the modified conjugate function to $I(r,\cdot)$, be defined by
\begin{equation*}
  J(r,s)=\sup_{t}[st-I(r,t)]\ \ \text{if}\ \ s\ge 0~;\ \ J(r,s)=0\ \ \text{if}\ \ s<0.
\end{equation*}
We further require the following:
\begin{itemize}
  \item[$(a2)$] For each $r>0$, $i(r,t)=0$ if $t\le0$, and $i(r,t)$ is strictly increasing (w.r.t.~t) in $[0,+\infty)$.
\item [$(a3)$] For each $d>0$, there exist $\delta_0\in(0,1)$ and $\delta_1>0$ such that
  \begin{equation*}
    I(r,t)\le \delta_0 i(r,t)t+\delta_1i(r,t), \ \ \  \ \forall~t>0, \ \forall~0<r\le d,
  \end{equation*}
  \item[$(a4)$] For all $\tau>0$, there holds
  \begin{equation*}
    \lim_{t\to+\infty}i(r,t)e^{-\tau t}=0, \ \ \forall\ r>0.
  \end{equation*}
\end{itemize}

\begin{remark}
we give some remarks on the above assumptions.
\begin{itemize}
  \item[1)] We restrict that $f$ and $g$ are both non-decreasing, which is believed to include the cases of primary interest(see \cite{BF1}). Note that the profile function $g$ may has a simple discontinuity. This case corresponds to a jump in vorticity
at the boundary of the cross-section of the vortex ring.
  \item[2)] The assumption $(a2)$ rules out the case that $f\equiv0$ and $g$ is a step function. This case has been studied in \cite{Cao2,FT}. However, our approach can also apply to this situation with some modifications. For the sake of consistency, we will not consider this case here.
  \item [3)] The assumption $(a3)$ can be regarded as a weakened type of the Ambrosetti-Rabinowitz condition, cf. condition $(p5)$ in \cite{AR}. It implies $i(r,t)\to +\infty$ as $t\to +\infty$, see \cite{Ni}. Note that $i(r,\cdot)$ and $\partial_s J(r,\cdot)$ are inverse graphs (see \cite{Roc}), one can verify that $(a3)$ is actually equivalent to
  \begin{equation*}
    \noindent(a3)':\ \ \ \  \   J(r,s)\ge (1-\delta_0)\partial_sJ(r,s)s-\delta_1s, \ \ \ \forall~t>0, \ \forall~0<r\le d.\ \ \ \ \  \ \ \ \ \ \
  \end{equation*}
  \item [4)] The assumption $(a4)$ relaxes the condition (1.6) in \cite{Ber}. Appropriate growth conditions are usually necessary in this type of problem.
  \item [5)] It is easy to find some $(f,g)$ satisfying $(a1)$-$(a4)$. For example, we have
  \begin{itemize}
    \item [(i)]$f\equiv 0$ and $g(t)=t_+^p$ for $p>0$. In such a case, one has
    \begin{equation*}
   \ \ \  \ \ \   i(r,t)=t_+^p, \ \ I(r,t)=\frac{1}{p+1}t_+^{p+1};\ \ J(r,s)=\frac{p}{p+1}s_+^{1+\frac{1}{p}},\ \ \partial_sJ(r,s)=s_+^{\frac{1}{p}}.
    \end{equation*}
     As mentioned above, de Valeriola and Van Schaftingen \cite{DV} has studied this case for $p>1$. If $p=1$, this model also occurs in the plasma problem, see, e.g., \cite{Ber, CF1, Te}. It describes the equilibrium of a plasma confined in a toroidal cavity (a ``Tokamak machine''). In \cite{CF1}, Caffarelli and Friedman studied asymptotic behavior of this problem in a bounded planar domain.
    \item [(ii)] $f(t)=t_+$ and $g(t)=\alpha \chi_{_{\{t>0\}}}$ for some $\alpha\ge0$. In such a case, one has
    \begin{equation*}
    \begin{split}
         i(r,t)&=\alpha \chi_{_{\{t>0\}}}+\frac{t_+}{r^2}, \ \ \ \ I(r,t)=\alpha t_++\frac{t_+^2}{2r^2}~; \\
          J(r,s)&=\frac{1}{2}(s-\alpha)_+^2r^2,\ \ \ \ \partial_sJ(r,s)=(s-\alpha)_+r^2.
    \end{split}
    \end{equation*}
     As mentioned above, Turkington \cite{Tur89} was concerned with this case.
    \item[(iii)] $f(t)=t_+^p$ for $p>0$ and $g(t)\equiv0$. In such a case, one has
     \begin{equation*}
    \begin{split}
         i(r,t)&=\frac{t_+^p}{r^2}, \ \ \ \ \ \ \  I(r,t)=\frac{1}{r^2}\frac{t_+^{p+1}}{p+1}~; \\
          J(r,s)&=\frac{pr^{\frac{2}{p}}}{p+1}s_+^{1+\frac{1}{p}},\ \ \ \  \partial_sJ(r,s)=\left({r^2s_+}\right)^\frac{1}{p}.
    \end{split}
    \end{equation*}
    We remark that in this case the flow is a Beltrami flow (see \cite{Abe, MB, Wu}).
    \item [(iv)] $f(t)=g(t)=t_+^p$ for $p>0$. In such a case, one has
    \begin{equation*}
    \begin{split}
         i(r,t)&=\left(1+\frac{1}{r^2}\right)t_+^p, \ \ \ \ I(r,t)=\left(1+\frac{1}{r^2}\right)\frac{t_+^{p+1}}{p+1}~; \\
          J(r,s)&=\frac{p}{p+1}\left(\frac{r^2}{r^2+1}\right)^\frac{1}{p}s_+^{1+\frac{1}{p}},\ \ \ \  \partial_sJ(r,s)=\left(\frac{r^2s}{r^2+1}\right)^\frac{1}{p}.
    \end{split}
    \end{equation*}
     To the best of our knowledge, there are no desingularization results in $\mathbb{R}^3$ for this case.
  \end{itemize}
\end{itemize}
    Of course, there are many more generators $(f,g)$ satisfying $(a1)$-$(a4)$. The situation we consider here covers almost all previously known cases.
\end{remark}

Now, we turn to state the main result. To this end, we need to introduce some notation.
Let $\Pi=\{(r,z)~|~r>0, z\in\mathbb{R}\}$ denote a meridional half-plane ($\theta$=constant); $d\nu=rdrdz$ is a measure on $\Pi$.

We define $\mathcal{K}$, the inverse of $\mathcal{L}$, as follows. One can check that the operator $\mathcal{K}$ is well-defined; see, e.g., \cite{BB}.
\begin{definition}
The Hilbert space $\mathcal{H}$ is the completion of $C_0^\infty(\Pi)$ with the scalar products
\begin{equation*}
  \langle u,v\rangle_\mathcal{H}=\int_\Pi\frac{1}{r^2}\nabla u\cdot\nabla v d\nu.
\end{equation*}
We define inverses $\mathcal{K}$ for $\mathcal{L}$ in the weak solution sense: $\mathcal{K}u\in \mathcal{H}$ and
\begin{equation}\label{2-1}
  \langle \mathcal{K}u,v\rangle_\mathcal{H}=\int_\Pi uv d\nu \ \ \ for\  all\  v\in \mathcal{H}, \ \ when \ u \in  L^{10/7}(\Pi,r^3drdz).
\end{equation}
\end{definition}

As in \cite{DV}, for an axisymmetric set $A\subseteq \mathbb{R}^3$ we introduce the axisymmetric distance as follows
\begin{equation*}
  dist_{\mathcal{C}_r}(A)=\sup_{x\in A}\inf_{x'\in{\mathcal{C}_r}}|x-x'|,
\end{equation*}
where $\mathcal{C}_r:=\{x\in\mathbb{R}^3~|x_1^2+x_2^2=r^2,x_3=0\}$ for some $r>0$.

Our main result in this paper is as follows.

\begin{theorem}\label{thm1}
Let $\kappa>0$ and $W>0$ be two given numbers. Suppose $(-B',HH')$ satisfies $(a1)-(a4)$ and $H(0)=0$. Then for all sufficiently small $\varepsilon>0$, there exists a solution $(\psi^\varepsilon,\xi^\varepsilon,\zeta^\varepsilon)$  with the following properties:
\begin{itemize}
\item[(i)]For any $p>1$, $0<\gamma<1$, $\psi^\varepsilon\in W^{2,p}_{\text{loc}}(\Pi)$ and satisfies
\begin{equation*}
  \mathcal{L}\psi^\varepsilon=\zeta^\varepsilon\ \ \text{a.e.} \ \text{in} \ \Pi.
\end{equation*}

\item[(ii)] $(\psi^\varepsilon,\xi^\varepsilon,\zeta^\varepsilon)$ is of the form
\begin{equation*}
\begin{split}
    & \psi^\varepsilon=\mathcal{K}\zeta^\varepsilon-\frac{Wr^2}{2}\log{\frac{1}{\varepsilon}}-\mu^\varepsilon,\ \ \xi^\varepsilon=\frac{1}{\varepsilon}H(\psi^\varepsilon),\ \  \\
     &  \zeta^\varepsilon=\frac{1}{\varepsilon^2}\left(-B'(\psi^\varepsilon)+\frac{1}{r^2}H(\psi^\varepsilon)H'(\psi^\varepsilon)\right),
\end{split}
\end{equation*}
for some $\mu^\varepsilon>0$ depending on $\varepsilon$. Furthermore, we have
\begin{equation*}
  \psi^\varepsilon \le C,\ \ \ \xi^\varepsilon\le C/\varepsilon,\ \ \ \zeta^\varepsilon\le C/\varepsilon^2,
\end{equation*}
for some positive constant $C$ independent of $\varepsilon$.

 \item[(iii)]$\zeta^\varepsilon$ and $\zeta^\varepsilon$ both have compact supports in $\Pi$ and $\int_\Pi \zeta^\varepsilon d\nu\equiv \kappa$. In addition, $\psi^\varepsilon$,  $\xi^\varepsilon$ and $\zeta^\varepsilon$ are symmetric decreasing in $z$. Also,~ $\partial_z\psi^\varepsilon(r,z)<0$ for all $r$, $z>0$.

 \item[(iv)]
 There exist some constants $R_0, R_1>0$ independent of $\varepsilon$ such that
 \begin{equation*}
   R_0\varepsilon \le diam\left(supp(\zeta^\varepsilon)\right)\le R_1\varepsilon.
 \end{equation*}
Moreover, as $\varepsilon \to 0^+$, we have
\begin{equation*}
  dist_{\mathcal{C}_{r_{_*}}}\left(supp(\zeta^\varepsilon)\right)  \to 0,\ \ \text{with}\ \ r_*=\frac{\kappa}{4\pi W},
\end{equation*}
and
\begin{equation*}
  \mu^\varepsilon  =\frac{3\kappa^2}{32\pi^2W}\log{\frac{1}{\varepsilon}}+O(1).
\end{equation*}
\item[(v)] There holds
\begin{equation*}
  \mathbf{v}^\varepsilon=\frac{1}{r}\left(-\frac{\partial\psi^\varepsilon}{\partial z}\mathbf{e}_r+\xi^\varepsilon\mathbf{e}_\theta+\frac{\partial\psi^\varepsilon}{\partial r}\mathbf{e}_z\right)\to -W\log\frac{1}{\varepsilon}~\mathbf{e}_z\ \  \text{at}\ \infty.
\end{equation*}
Moreover, as $r\to 0$,
\begin{equation*}
     \frac{1}{r}\frac{\partial\psi^\varepsilon}{\partial z}\to 0,\ \ \frac{\xi^\varepsilon}{r}\to 0\ \text{and}~ \ \frac{1}{r}\frac{\partial\psi^\varepsilon
     }{\partial r}\  \text{approaches a finite limit}.
\end{equation*}
In addition, the vortex core $supp(\zeta^\varepsilon)$ is a topological disc in $\Pi$.
\end{itemize}
\end{theorem}

\begin{remark}

Now, let us give some remarks on the above result.

\begin{itemize}
  \item Kelvin and Hicks showed that if the vortex ring with circulation $\kappa$ has radius $r_*$ and its cross-section $\varepsilon$ is small, then the vortex ring moves at the velocity\,(see, e.g., \cite{Lamb, Wu})
\begin{equation*}
  \frac{\kappa}{4\pi r_*}\Big(\log \frac{8r_*}{\varepsilon}-\frac{1}{4}\Big).
\end{equation*}
One can see that our result is consistent with this Kelvin-Hicks formula.
  \item If $B'\equiv 0$, then by \eqref{1-8} we have
  \begin{equation*}
  \pmb{\omega}=\frac{H'(\psi^\varepsilon)}{\varepsilon}\mathbf{v}^\varepsilon.
  \end{equation*}
   Consequently, these vortex rings are  actually a family of Beltrami flows (see \cite{Abe, MB, Wu}). As a particular kind of solutions of the Euler equations,
   Beltrami flow has some special properties, and it is one of the most interesting research objects of fluid mechanics. We refer the interested readers to \cite{Abe} for some recent results in this direction. We also point out that the steady solutions from the above theorem are actually traveling wave solutions to the time-dependent Euler equations (cf.\cite{Abe}).
  \item The axisymmetric ideal magnetohydrodynamic equilibria with incompressible flows is governed by a generalized Grad-Shafranov equation which is of the form (see, e.g., \cite{Tas, Throu})
      \begin{equation*}
        \mathcal{L}\psi=g(\psi)+\frac{f(\psi)}{r^2}+r^2h(\psi)
      \end{equation*}
      for some given functions $f$, $g$ and $h$. Our approach can be applied to this problem in a straightforward way. Indeed, one can expect to extend our method to some more general semilinear elliptic equations.
\end{itemize}

 It is instructive to compare the known results with ours. Let $(g,f)=(-B', HH')$.

 First, let us discuss the non-swirling case. Note that, in this situation, $f\equiv0$.
\begin{itemize}
  \item  When $g(t)=t_+^p$, $p>1$, de Valeriola and Van Schaftingen have obtained similar result, cf. Theorem 1 in \cite{DV}. Several types of domains were considered in \cite{DV}. If there are walls or obstacles, the geometry of domain will affect the asymptotic localization of vortex rings. We remark that our method can also be applied to more general domains as in \cite{Cao2, DV}.
  \item  When $g(t)$ is the Heaviside function, some desingularization results were established in \cite{Cao2, FT}. In \cite{FT}, the velocity at infinity of the flow arises as a Lagrange multiplier and hence is left undetermined. Moreover, their method seems to rely in essential way on the connectness of the vortex core, which is hard to prove. This defect has been improved by Cao et al. \cite{Cao2}.   Our work can be viewed as a continuation of the recent work \cite{Cao2, Cao4}.
  \item  When $g(t)=t_+$,  Caffarelli and Friedman \cite{CF1} studied asymptotic behavior of this problem in a bounded planar domain. They constructed a
family of plasmas which were shown to converge to the part of the boundary of the domain, which is the farthest away from the $z$-axis. Applying our method, we may obtain similar results as in \cite{CF1}.
\end{itemize}

Second, we consider the swirling case. There have been very few results in this aspect.
\begin{itemize}
  \item Tadie's works \cite{Ta, Ta2} were based on the variational approach proposed in \cite{BF1}. For the variational solutions constructed in this way, the vortex-strength parameter will arise as a Lagrange multiplier and hence is left undetermined. Hence asymptotic behavior of
those solutions could not be studied precisely. Tadie studied the asymptotic behavior by letting the flux constant $\mu$ diverge. Moreover, in \cite{Ta2}, the domain was required to be away from the $z$-axis. This requirement avoids the singularity of the operator $\mathcal{L}$. However, it seems to rule out the situation of primary interest in the study of vortex rings.
  \item Turkington \cite{Tur89} constructed a two-parameter family of desingularized steady solutions $\psi^{\varepsilon,\alpha}$ for
\begin{equation*}
  f(t)=t_+, \ \ g^{\varepsilon,\alpha}(t)=\varepsilon \alpha \chi_{_{\{t>0\}}},
\end{equation*}
where $\varepsilon>0$ and $\alpha\ge 0$. Although $g^{\varepsilon,\alpha}$ here depends on the parameter $\varepsilon$, we note that our approach apply as well without any significant changes. Turkington has remarked in \cite{Tur89} for some possible generalizations of his results. However, the conditions therein seem to be quite involved. Our present work may be thought of as an extension of Turkington’s work. We greatly refined his results here. Our conditions differ from his and may contain a wider class of solutions. Moreover, the method we adopt here is different from his. Finally, as discussed in \cite{Tur89}, it is also interesting to investigate the asymptotic behavior of solutions with respect to two parameters. We leave these extensions to the reader.
\end{itemize}

\end{remark}

Comparing with known results, we provide a broader class of vortex rings. Our work actually dose shed a light on the study for this problem.

Having constructed the above vortex rings, we are interested in their stability/nonstability properties. Very recently, Choi \cite{Choi} studied nonlinear stability of Hill's spherical vortex in axi-symmetric perturbations.  We refer the reader to \cite{Choi, Lif, Pro0, Pro} and the references therein for some discussions in this direction. In addition, the (local) uniqueness of the solutions from the above theorem seems to be a difficult open problem. As far as we know, only some extremely special cases have been proved in this direction; see \cite{AF, AF2, Fra}. These topics is the aim of ongoing work.

The time evolution of vortex rings is also a matter of concern. An outstanding open question is the vortex filament conjecture, which states that one can find solutions of the Euler equations for which the vorticity remains close for a significant period of time to a given curve evolving by binormal curvature flow. We refer the reader to \cite{Bene, But, Dav2, Je} and references therein for some works on this problem.

The Bragg-Hawthorne equation also plays an important role in the study of vortex breakdown, see, e.g., \cite{BS, Lu2001, WR}). Our results may also enhance our understanding of axisymmetric vortex breakdown. Finally, we remark that there is a similar situation with similar results in the study of vortex pairs
for the two-dimensional Euler equation; see, e.g., \cite{ CLW, Cao1, Cao3, Dav1, Sm, Tur83}.

\section{Vortex Rings in $\mathbb{R}^3$ and Proof of Theorem \ref{thm1}}
In this section, we will provide the proof of Theorem \ref{thm1}. For the sake of clarity, we split the proof into several lemmas.

Throughout the sequel we shall use the following notations: Let $\Pi=\{(r,z)~|~r>0, z\in\mathbb{R}\}$ denote a meridional half-plane ($\theta$=constant); $B_\delta(y)$ denotes an open ball in $\mathbb{R}^2$ of center $y$ and radius $\delta>0$; Let $\chi_{_A}$ denote the characteristic function of $A\subseteq\mathbb{R}^2$; $h_+$ denotes the positive parts of $h$; Lebesgue measure on $\mathbb{R}^N$ is denoted $\textit{m}_N$, and is to be understood as the measure defining any $L^p$ space, $W^{1,p}$ space and $W^{2,p}$ space, except when stated otherwise; $\nu$ denotes the measure on $\Pi$ having density $r$ with respect to $\textit{m}_2$ and $|\cdot|$ denotes the $\nu$-measure.

Let $K(r,z,r',z')$ be the Green's function of $\mathcal{L}$ in $\Pi$, with respect to zero Dirichlet data and measure $\nu$ (see, e.g., \cite{BB, BF1, Ta}). Then
\begin{equation*}
   K(r,z,r',z')=\frac{rr'}{4\pi}\int_{-\pi}^{\pi}\frac{\cos\theta'd\theta'}{[(z-z')^2+r^2+r'^2-2rr'\cos\theta']^\frac{1}{2}}.
\end{equation*}
One can easily show that the operator $\mathcal{K}$ is an integral operator with kernel $K(r,z,r',z')$ for the case considered here. We shall use this Green's representation formula directly without further explanation.

To begin with, we need some estimates for the Green’s function $K$. From \cite{Ta}, we have
\begin{lemma}\label{le1}
 Let
\begin{equation}\label{2-1}
 \sigma=[(r-r')^2+(z-z')^2]^\frac{1}{2}/(4rr')^\frac{1}{2},
\end{equation}
then for all $\sigma > 0$
\begin{equation}\label{Tadiewrong}
  0<K(r,z,r',z')\leq\frac{(rr')^\frac{1}{2}}{4\pi}\sinh^{-1}(\frac{1}{\sigma}).
\end{equation}
\end{lemma}

We have the following expansion, see also \cite{Bene}.
\begin{lemma}\label{le2}
  Let $\sigma$ be defined by \eqref{2-1}, then there exists a continuous function $l\in L^{\infty}(\Pi\times \Pi)$ such that
\begin{equation}\label{2-2}
   K(r,z,r',z')=\frac{\sqrt{rr'}}{2\pi}\log{\frac{1}{\sigma}}+\frac{\sqrt{rr'}}{2\pi}\log(1+\sqrt{\sigma^2+1})+l(r,z,r',z')\sqrt{rr'},\ \ \text{in}\ \Pi \times \Pi.
\end{equation}
\end{lemma}

\begin{proof}
 We have
\begin{equation}\label{2-3}
\begin{split}
    K(r,z,r',z') & =\frac{rr'}{2\pi}\int_{0}^{\pi}\frac{\cos\theta'd\theta'}{[(z-z')^2+r^2+r'^2-2rr'\cos\theta']^\frac{1}{2}}\\
                 & =\frac{\sqrt{rr'}}{2\pi}\int_{0}^{\pi}\frac{\cos\theta'd\theta'}{\{[(z-z')^2+(r-r')^2]/(rr')+2(1-\cos\theta')\}^\frac{1}{2}}\\
                 & =\frac{\sqrt{rr'}}{2\pi}\int_{0}^{\pi}\frac{[\cos\theta'/2+(\cos\theta'-\cos\theta'/2)]d\theta'}{[4\sigma^2+4(\sin\theta'/2)^2]^\frac{1}{2}}.
\end{split}
\end{equation}
Notice that
\begin{equation}\label{2-4}
  \int_{0}^{\pi}\frac{\cos\theta'/2d\theta'}{[4\sigma^2+4(\sin\theta'/2)^2]^\frac{1}{2}}=\log(1+\sqrt{\sigma^2+1})+\log{\frac{1}{\sigma}},
\end{equation}
and
\begin{equation}\label{2-5}
  \int_{0}^{\pi}\frac{|\cos\theta'-\cos\theta'/2|d\theta'}{[4\sigma^2+4(\sin\theta'/2)^2]^\frac{1}{2}}\le \text{const.}<+\infty.
\end{equation}
From \eqref{2-3},\eqref{2-4} and \eqref{2-5}, $\eqref{2-2}$ clearly follows.
\end{proof}

\subsection{Variational problem} Our focus is on the asymptotics of solutions of
\begin{equation}\label{aim2}
  \mathcal{L}\psi^\varepsilon=\frac{1}{\varepsilon^2}\left(g(\psi^\varepsilon)+\frac{f(\psi^\varepsilon)}{r^2}\right)=\zeta^\varepsilon,\ \ \ \ \   \int_\Pi \zeta^\varepsilon d\nu=\kappa.
\end{equation}
Roughly speaking, there are two methods to study the problem of steady vortex rings, which are the stream-function method and the vorticity method. The stream-function method is to find a solution of \eqref{aim2} with the desired properties; see, e.g., \cite{Abe, AS, DV, BF1, Ni, Ta, Ta2, Yang2}. The vorticity method is to solve a variational problem for the potential vorticity $\zeta$; see, e.g., \cite{BB, Benj, Cao2, FT, Tur89}. In contrast with the stream-function method, the vorticity method has strong physical motivation. Mathematically, the vorticity method can be regarded as a dual variational principle; see \cite{Am, Ber, Cao3, ET, St} for example.

 In this paper, we consider this problem by using an improved vorticity method.

For fixed $W>0$ and $\kappa>0$, we consider the energy functional given by
\begin{equation*}
  \mathcal{E}_\varepsilon(\zeta)=\frac{1}{2}\int_\Pi{\zeta \mathcal{K}\zeta}d\nu-\frac{{W}}{2}\log{\frac{1}{\varepsilon}}\int_{\Pi}r^2\zeta d\nu-\frac{1}{\varepsilon^2}\int_\Pi J(r,\varepsilon^2\zeta)d\nu.
\end{equation*}

Let $$D:=\{(r,z)\in \Pi~|~\frac{r_*}{2}<r<2r_*,\  -1<z<1\}, \ \ \ \text{where}\   ~r_*:=\frac{\kappa}{4\pi W},$$
and
\begin{equation*}
\mathcal{A}_{\varepsilon,\Lambda}=\{\zeta\in L^\infty(\Pi)~|~ 0\le \zeta \le \frac{\Lambda}{\varepsilon^2}~ \text{a.e.}, \int_{\Pi}\zeta d\nu \le\kappa,~ supp(\zeta)\subseteq D \},
\end{equation*}
 where $\varepsilon>0$ is a small parameter, $g(0^+)=\lim_{t\to 0^+}g(t)$ and $\Lambda>\max\{1,g(0^+)\}$ is an appropriate constant (whose value will be determined later, see Lemma \ref{le14} below). We will seek maximizers of $\mathcal{E}_\varepsilon$ relative to $\mathcal{A}_{\varepsilon,\Lambda}$.

For any $\zeta\in \mathcal{A}_{\varepsilon, \Lambda}$, one can easily see that $\mathcal{K}\zeta \in W^{2,q}_{\text{loc}}(\Pi)$ for any $q>1$.

Let ${\zeta}^*$ be the Steiner symmetrization of $\zeta$ with respect to the line $z=0$ in $\Pi$ (see Appendix I of \cite{BF1}).

An absolute maximum for $\mathcal{E}_\varepsilon$ over $\mathcal{A}_{\varepsilon, \Lambda}$  can be easily found.

\begin{lemma}\label{le3}
There exists $\zeta=\zeta^{\varepsilon, \Lambda} \in \mathcal{A}_{\varepsilon,\Lambda}$ such that
\begin{equation*}
 \mathcal{E}_\varepsilon(\zeta^{\varepsilon, \Lambda})= \max_{\tilde{\zeta} \in \mathcal{A}_{\varepsilon,  \Lambda}}\mathcal{E}_\varepsilon(\tilde{\zeta})<+\infty.
\end{equation*}
Moreover, $\zeta^{\varepsilon,\Lambda}=(\zeta^{\varepsilon, \Lambda})^*$.
\end{lemma}

\begin{proof}
We may take a sequence $\{\zeta_{k}\}\subset \mathcal{A}_{_{\varepsilon,\Lambda}}$ such that as $j\to +\infty$
\begin{equation*}
  \begin{split}
        \mathcal{E}_\varepsilon(\zeta_{k}) & \to \sup\{\mathcal{E}_\varepsilon(\tilde{\zeta})~|~\tilde{\zeta}\in \mathcal{A}_{_{\varepsilon,\Lambda}}\}, \\
        \zeta_{k} & \to \zeta\in L^{12/7}({\Pi,\nu})~~\text{weakly}.
  \end{split}
\end{equation*}
It is easy to verify that $\zeta\in \mathcal{A}_{_{\varepsilon,\Lambda}}$. Using the standard arguments (see \cite{BB}), we may assume that $\zeta_{k}=(\zeta_{k})^*$, and hence $\zeta=\zeta^*$. By Lemma 2.12 of \cite{BB}, we first have
\begin{equation*}
      \lim_{k\to +\infty}\int_\Pi{\zeta_{k} \mathcal{K}\zeta_{k}}d\nu = \int_\Pi{\zeta \mathcal{K}\zeta}d\nu,\ \text{as}\ k\to +\infty.
\end{equation*}
On the other hand, we have the lower semicontinuity of the rest of terms, namely,
\begin{equation*}
  \begin{split}
    \liminf_{k\to +\infty} \int_{\Pi}r^2\zeta_k d\nu  & \ge  \int_{\Pi}r^2\zeta d\nu, \\
    \liminf_{k\to +\infty}\int_\Pi J(r,\varepsilon^2\zeta_k)d\nu   & \ge \int_\Pi J(r,\varepsilon^2\zeta)d\nu.
  \end{split}
\end{equation*}
Consequently, we conclude that $ \mathcal{E}_\varepsilon(\zeta)=\lim_{k\to +\infty} \mathcal{E}_\varepsilon(\zeta_k)=\sup \mathcal{E}_\varepsilon$, with $\zeta\in \mathcal{A}_{_{\varepsilon, \Lambda}}$, which completes the proof.
\end{proof}

The following lemma gives the profile of $\zeta^{\varepsilon, \Lambda}$, which is useful for the rest of the analysis.
\begin{lemma}\label{le4}
Let $\zeta^{\varepsilon, \Lambda}$ be a maximizer as in Lemma \ref{le3}, then there exists a Lagrange multiplier $\mu^{\varepsilon, \Lambda} \ge 0$ such that
\begin{equation}\label{2-6}
\zeta^{\varepsilon, \Lambda}=\frac{1}{ \varepsilon^2}i\left(r,\psi^{\varepsilon, \Lambda}\right){\chi}_{_{A_{\varepsilon, \Lambda}}}+\frac{\Lambda}{\varepsilon^2}{\chi}_{_{B_{\varepsilon, \Lambda}}} \ \ a.e.\  \text{in}\  D,
\end{equation}
where
\begin{equation}\label{2-7}
 \psi^{\varepsilon, \Lambda}=\mathcal{K}\zeta^{\varepsilon, \Lambda}-\frac{Wr^2}{2}\log{\frac{1}{\varepsilon}}-\mu^{\varepsilon,\Lambda},
\end{equation}
and
\begin{equation*}
\begin{split}
   A_{\varepsilon,\Lambda}:= &  \big{\{}(r,z)\in \Pi \mid 0<\psi^{\varepsilon, \Lambda}<\partial_sJ(r,\Lambda)\big{\}},   \\
      B_{\varepsilon, \Lambda}:= &  \big{\{}(r,z)\in \Pi \mid \psi^{\varepsilon, \Lambda}\ge \partial_sJ(r,\Lambda)\big{\}}.
\end{split}
\end{equation*}
 Furthermore, whenever $\mathcal{E}_\varepsilon(\zeta^{\varepsilon, \Lambda})>0$ and $\mu^{\varepsilon, \Lambda}>0$ there holds $\int_D\zeta^{\varepsilon, \Lambda} d\nu=\kappa$.
\end{lemma}

\begin{proof}
Note that we may assume $\zeta^{\varepsilon,\Lambda} \not\equiv 0$, otherwise the conclusion is obtained by letting $\mu^{\varepsilon, \Lambda}=0$.
We consider a family of variations of $\zeta^{\varepsilon, \Lambda}$
\begin{equation*}
  \zeta_{(t)}=\zeta^{\varepsilon, \Lambda}+t(\tilde{\zeta}-\zeta^{\varepsilon, \Lambda}),\ \ \ t\in[0,1],
\end{equation*}
defined for arbitrary $\tilde{\zeta}\in \mathcal{A}_{\varepsilon, \Lambda}$. Since $\zeta^{\varepsilon, \Lambda}$ is a maximizer, we have
\begin{equation*}
  \begin{split}
     0 & \ge \frac{d}{dt}\mathcal{E}_\varepsilon(\zeta_{(t)})|_{t=0^+} \\
       & =\int_{D}(\tilde{\zeta}-\zeta^{\varepsilon, \Lambda})\left[\mathcal{K}\zeta^{\varepsilon, \Lambda}-\frac{Wr^2}{2} \log{\frac{1}{\varepsilon}}-\partial_sJ(r,\varepsilon^2\zeta^{\varepsilon, \Lambda}) \right]d\nu.
  \end{split}
\end{equation*}
This implies that for any  $\tilde{\zeta}\in \mathcal{A}_{\varepsilon,\Lambda}$, there holds
\begin{equation*}
\begin{split}
    \int_{D}\zeta^{\varepsilon,\Lambda} &\left[\mathcal{K}\zeta^{\varepsilon, \Lambda}-\frac{Wr^2}{2} \log{\frac{1}{\varepsilon}}-\partial_sJ(r,\varepsilon^2\zeta^{\varepsilon, \Lambda})\right]d\nu  \\
     & \ \ \ \ \ \ \ \ \ \ \ \ \ \  \ge \int_{D}\tilde{\zeta}  \left[\mathcal{K}\zeta^{\varepsilon, \Lambda}-\frac{Wr^2}{2} \log{\frac{1}{\varepsilon}}-\partial_sJ(r,\varepsilon^2\zeta^{\varepsilon, \Lambda})\right]d\nu.
\end{split}
\end{equation*}
By an adaptation of the bathtub principle (see Lemma \ref{bathtub}), we obtain that for any point in $D$, it holds
\begin{align}\label{2-8}
\begin{cases}
    ~\mathcal{K}\zeta^{\varepsilon, \Lambda}-\frac{Wr^2}{2} \log{\frac{1}{\varepsilon}}- \mu^{\varepsilon, \Lambda} \ge &\partial_sJ(r,\varepsilon^2\zeta^{\varepsilon, \Lambda})\ \ \ \ \  \mbox{whenever}\  \zeta^{\varepsilon, \Lambda}=\frac{\Lambda}{\varepsilon^2}, \\
    ~\mathcal{K}\zeta^{\varepsilon, \Lambda}-\frac{Wr^2}{2} \log{\frac{1}{\varepsilon}}- \mu^{\varepsilon, \Lambda}= & \partial_sJ(r,\varepsilon^2\zeta^{\varepsilon, \Lambda})\ \ \ \ \  \mbox{whenever}\  0<\zeta^{\varepsilon, \Lambda}<\frac{\Lambda}{\varepsilon^2}, \\
    ~\mathcal{K}\zeta^{\varepsilon, \Lambda}-\frac{Wr^2}{2} \log{\frac{1}{\varepsilon}}- \mu^{\varepsilon, \Lambda}\le  & \partial_sJ(r,\varepsilon^2\zeta^{\varepsilon, \Lambda}) \ \ \ \ \  \mbox{whenever}\  \zeta^{\varepsilon, \Lambda}=0,
\end{cases}
\end{align}
for some $\mu^{\varepsilon, \Lambda} \ge 0$. It follows that $\{0<\zeta^{\varepsilon, \Lambda}\le g(0^+)\varepsilon^{-2}\}\subseteq\{\mathcal{K}\zeta^{\varepsilon, \Lambda}-\frac{Wr^2}{2} \log{\frac{1}{\varepsilon}}=\mu^{\varepsilon, \Lambda}\}$. Since $\mathcal{L}(\mathcal{K}\zeta^{\varepsilon, \Lambda}-\frac{Wr^2}{2} \log{\frac{1}{\varepsilon}})=\zeta^{\varepsilon, \Lambda}$ almost everywhere in $D$, we conclude that $m_2(\{0<\zeta^{\varepsilon, \Lambda}\le g(0^+)\varepsilon^{-2}\})=0$. Recall that for $t>0$ and $s>g(0^+)$, there holds
\begin{equation*}
  t= \partial_sJ(r,s)\ \ \ \text{if and only if}\ \ \ s= i(r,t).
\end{equation*}
Now the stated form $\eqref{2-6}$ follows from \eqref{2-8} immediately. Moreover, when $\mathcal{E}_\varepsilon(\zeta^{\varepsilon, \Lambda})>0$ and $\mu^{\varepsilon, \Lambda}>0$, by Lemma \ref{bathtub} we have $\int_D\zeta^{\varepsilon,\Lambda} d\nu=\kappa$. The proof is thus complete.
\end{proof}

Note that our goal is to obtain
 \begin{equation*}
   \zeta^{\varepsilon, \Lambda}=\frac{1}{ \varepsilon^2}i(r,\psi^{\varepsilon, \Lambda}),\ \ \ \ \ a.e.\  \text{in}\  \Pi,
 \end{equation*}
which follows that $\psi^{\varepsilon, \Lambda}$ satisfies the equation \eqref{aim2}. To this end, we first need to eliminate the patch part in \eqref{2-6}. This can be done by choosing $\Lambda$ carefully, see Lemma \ref{le14} below. And then we also need to show that \eqref{2-6} actually holds in the whole plane $\Pi$, not just in the domain $D$. This can be achieved by the maximum principle, see Lemma \ref{le15} below.

\subsection{Asymptotic behavior}

Our focus next is the asymptotic behavior of $\zeta^{\varepsilon, \Lambda}$ when $\varepsilon \to 0^+$. In the sequel we shall denote $C,C_1,C_2, ..., $ for positive constants independent of $\varepsilon$ and $\Lambda$. For the sake of distinction, we will denote $C^*,C_1^*,C_2^*, ..., $  for positive constants which may depend on $\Lambda$, but not on $\varepsilon$. We shall use $o_\varepsilon(\Lambda)$ to denote the number which satisfies $o_\varepsilon(\Lambda)\to0$ as $\varepsilon \to 0^+$ for each fixed $\Lambda$.

The following lemma gives a lower bound of the energy.
\begin{lemma}\label{le5}
  For any  $a\in(r_*/2,2r_*)$, there exists $C>0$ such that
\begin{equation*}
  \mathcal{E}_\varepsilon(\zeta^{\varepsilon, \Lambda})\ge \left(\frac{a\kappa^2}{4\pi}-\frac{\kappa Wa^2}{2}\right)\log{\frac{1}{\varepsilon}}-C.
\end{equation*}
\end{lemma}

\begin{proof}
The key is to choose a suitable test function. Let $\tilde{\zeta} \in \mathcal{A}_{\varepsilon,\Lambda}$ be the characteristic function of the ball centered at $(a,0)$ with radius ${{\varepsilon}{\sqrt{\kappa /a\pi}}}$. By Lemmas $\ref{le1}$ and $\ref{le2}$, we obtain
 \begin{equation*}
\begin{split}
    \mathcal{E}_\varepsilon(\tilde{\zeta})&= \frac{1}{2}\iint_{D\times D}\tilde{\zeta}(r,z)K(r,z,r',z')\tilde{\zeta}(r',z')r'rdr'dz'drdz-{\frac{W}{2} \log{\frac{1}{\varepsilon}}}\int_{D}r^2\tilde{\zeta}d\nu\\
                    &\ \ \ \ \ -\frac{1}{\varepsilon^2}\int_D J(r, \varepsilon^2\tilde{\zeta})d\nu\\
                    &\ge \frac{a+O(\varepsilon)}{4\pi}\iint_{D\times D}\log{\frac{1}{[(r-r')^2+(z-z')^2]^{1/2}}}\tilde{\zeta}(r,z)\tilde{\zeta}(r',z')r'rdr'dz'drdz\\
                    &\ \ \ \ \ -\frac{\kappa W[a+O(\varepsilon)]^2}{2}\log{\frac{1}{\varepsilon}}-C\\
                    &\ge \left(\frac{a\kappa^2}{4\pi}-\frac{\kappa W a^2}{2}\right)\log{\frac{1}{\varepsilon}}-C.
\end{split}
\end{equation*}
Since $\zeta^{\varepsilon, \Lambda}$ is a maximizer, we have $\mathcal{E}_\varepsilon(\zeta^{\varepsilon, \Lambda})\ge \mathcal{E}_\varepsilon(\tilde{\zeta})$ and the  proof is complete.
\end{proof}

We now turn to estimate the Lagrange multiplier $\mu^{\varepsilon, \Lambda}$.

\begin{lemma}\label{le6}
There holds
  \begin{equation*}
  \begin{split}
      \mu^{\varepsilon,\Lambda} & \ge \frac{2\mathcal{E}_\varepsilon(\zeta^{\varepsilon, \Lambda})}{\kappa} +\frac{W}{2\kappa}\log\frac{1}{\varepsilon}\int_Dr^2 \zeta^{\varepsilon, \Lambda}  d\nu-|2\delta_0-1|\kappa^{-1}\int_D\partial_sJ(r,\Lambda)\zeta^{\varepsilon,\Lambda}d\nu \\
       &\ \ \ \ \  -(C+o_\varepsilon(\Lambda))\left(1+\left(\int_D \partial_sJ(r,\Lambda)\zeta^{\varepsilon,\Lambda}d\nu\right)^\frac{1}{2}\right) .
  \end{split}
  \end{equation*}
\end{lemma}

\begin{proof}
For convenience, let us abbreviate $(\zeta^{\varepsilon,\Lambda},\psi^{\varepsilon,\Lambda}, \mu^{\varepsilon, \Lambda})$ to $(\zeta, \psi,\mu)$ here. Recall that from $(a3)'$
\begin{equation*}
   J(r,s)\ge (1-\delta_0)\partial_sJ(r,s)s-\delta_1s, \ \ \ \forall~t>0, \ \forall~0<r\le 2r_*.
\end{equation*}
 By \eqref{2-8} and $(a3)'$, we then have
\begin{equation*}
  \begin{split}
     2\mathcal{E}_\varepsilon(\zeta) &=\int_D{\zeta \mathcal{K}\zeta}d\nu-{{W}}\log{\frac{1}{\varepsilon}}\int_{D}r^2\zeta d\nu-\frac{2}{\varepsilon^2}\int_D J(r,\varepsilon^2\zeta)d\nu\\
       &= \int_D{\zeta \left(\mathcal{K}\zeta-\frac{W}{2}r^2\log{\frac{1}{\varepsilon}}-\mu\right)}d\nu-\frac{2}{\varepsilon^2}\int_D J(r,\varepsilon^2\zeta)d\nu\\
       &\ \ \ \ \ \ \ \ \ \ \ \ -{\frac{W}{2}\log{\frac{1}{\varepsilon}}}\int_{D}r^2 \zeta d\nu+\mu\int_D\zeta d\nu \\
       & \le \int_D \zeta \psi d\nu-2(1-\delta_0)\int_D\partial_sJ(r,\varepsilon^2\zeta)\zeta d\nu+2\delta_1\int_D\zeta d\nu     \\
       &\ \ \ \ \ \ \ \ \ \ \ \ -{\frac{W}{2}\log{\frac{1}{\varepsilon}}}\int_{D}r^2 \zeta d\nu+\mu\int_D\zeta d\nu \\
  \end{split}
\end{equation*}
Thus we get
\begin{equation}\label{2-10}
  2\mathcal{E}_\varepsilon(\zeta)\le \int_D \zeta \psi d\nu-2(1-\delta_0)\int_D\partial_sJ(r,\varepsilon^2\zeta)\zeta^{\varepsilon,\Lambda}d\nu-{\frac{W}{2}\log{\frac{1}{\varepsilon}}}\int_{D}r^2 \zeta d\nu +\mu\kappa+2\delta_1\kappa.
\end{equation}
If we take $\psi_+ \in \mathcal{H}$ as a test function, we then obtain
\begin{equation}\label{2-11}
  \int_\Pi \frac{|\nabla \psi_+|^2}{r^2} d\nu =\int_D \zeta \psi d\nu,
\end{equation}
Let $U:=\left(\psi-\partial_sJ(r,\Lambda)\right)_+$.
Recalling \eqref{2-8}, by H\"older's inequality and Sobolev's inequality, we have
\begin{equation}\label{2-12}
\begin{split}
 \int_D \zeta \psi d\nu & = \int_D \zeta \left(\psi-\partial_sJ(r,\Lambda)\right)_+ d\nu+\int_D\zeta \partial_sJ(r,\varepsilon^2\zeta)d\nu \\
     & \le  \frac{\Lambda}{\varepsilon^2}|\{\zeta=\Lambda \varepsilon^{-2}\}|^\frac{1}{2} \left(\int_D U^2 d\nu\right)^\frac{1}{2}+\int_D\zeta \partial_sJ(r,\varepsilon^2\zeta)d\nu \\
     & \le  \frac{C\Lambda}{\varepsilon^2}|\{\zeta=\Lambda \varepsilon^{-2}\}|^\frac{1}{2} \left(\int_D U^2 drdz\right)^\frac{1}{2}+\int_D\zeta \partial_sJ(r,\varepsilon^2\zeta)d\nu\\
     & \le  \frac{C\Lambda}{\varepsilon^2}|\{\zeta=\Lambda \varepsilon^{-2}\}|^\frac{1}{2} \left(\int_D |\nabla U|drdz+\int_D |U|drdz\right)+\int_D\zeta \partial_sJ(r,\varepsilon^2\zeta)d\nu\\
     & \le  \frac{C\Lambda}{\varepsilon^2}|\{\zeta=\Lambda \varepsilon^{-2}\}|^\frac{1}{2} \left(\int_{\{\zeta=\Lambda\varepsilon^{-2}\}}|\nabla \psi|drdz+\int_{\{\zeta=\Lambda\varepsilon^{-2}\}}|\nabla \partial_sJ(r,\Lambda)|drdz\right)\\
     &  \ \ \ \ \ \ \ \  \ \ \ \ \ \ \ \ \ \ +\frac{C\Lambda}{\varepsilon^2}|\{\zeta=\Lambda \varepsilon^{-2}\}|^\frac{1}{2}\int_{\{\zeta=\Lambda\varepsilon^{-2}\}} \psi drdz +\int_D\zeta \partial_sJ(r,\varepsilon^2\zeta)d\nu\\
     & \le  \frac{C\Lambda}{\varepsilon^2}|\{\zeta=\Lambda \varepsilon^{-2}\}| \left(\int_D \frac{|\nabla \psi|^2}{r^2} d\nu\right)^\frac{1}{2}+o_\varepsilon(\Lambda)\int_D \zeta \psi d\nu+o_\varepsilon(\Lambda)+\int_D\zeta \partial_sJ(r,\varepsilon^2\zeta)d\nu\\
     & \le C \left(\int_D \frac{|\nabla \psi_+|^2}{r^2} d\nu\right)^\frac{1}{2}+o_\varepsilon(\Lambda)\int_D \zeta \psi d\nu+o_\varepsilon(\Lambda)+\int_D\zeta \partial_sJ(r,\varepsilon^2\zeta)d\nu.
\end{split}
\end{equation}
Combining \eqref{2-11} and \eqref{2-12}, we get
\begin{equation*}
  \int_\Pi \frac{|\nabla \psi_+|^2}{r^2} d\nu\le C \left(\int_D \frac{|\nabla \psi_+|^2}{r^2} d\nu\right)^\frac{1}{2}+o_\varepsilon(\Lambda)\int_\Pi \frac{|\nabla \psi_+|^2}{r^2} d\nu+o_\varepsilon(\Lambda)+\int_D\zeta \partial_sJ(r,\varepsilon^2\zeta)d\nu
\end{equation*}
which implies
\begin{equation}\label{2-13}
  \int_D \zeta \psi d\nu\le (C+o_\varepsilon(\Lambda))\left(1+\left(\int_D\zeta \partial_sJ(r,\varepsilon^2\zeta)d\nu\right)^\frac{1}{2}\right)+\int_D\zeta \partial_sJ(r,\varepsilon^2\zeta)d\nu.
\end{equation}
Combining \eqref{2-10} and \eqref{2-13}, we then have
\begin{equation*}
\begin{split}
    2\mathcal{E}_\varepsilon(\zeta)\le  (C+o_\varepsilon(\Lambda))\left(1+\left(\int_D\zeta \partial_sJ(r,\varepsilon^2\zeta)d\nu\right)^\frac{1}{2}\right) &+(2\delta_0-1)\int_D\partial_sJ(r,\varepsilon^2\zeta)\zeta d\nu \\
      -{\frac{W}{2}\log{\frac{1}{\varepsilon}}}\int_{D}r^2 \zeta d\nu& +\mu\kappa+2\delta_1\kappa.
\end{split}
\end{equation*}
 Thus
\begin{equation*}
  \begin{split}
    2\mathcal{E}_\varepsilon(\zeta)\le  (C+o_\varepsilon(\Lambda))\left(1+\left(\int_D\zeta \partial_sJ(r,\Lambda)d\nu\right)^\frac{1}{2}\right) &+|2\delta_0-1|\int_D\partial_sJ(r,\Lambda)\zeta d\nu \\
      -{\frac{W}{2}\log{\frac{1}{\varepsilon}}}\int_{D}r^2 \zeta d\nu& +\mu\kappa+2\delta_1\kappa
\end{split}
\end{equation*}
which clearly implies the desired result.
The proof is thus complete.
\end{proof}

Combining Lemmas \ref{le5} and \ref{le6}, we immediately get the following estimate.
\begin{lemma}\label{le7}
For any  $a\in(r_*/2,2r_*)$, we have
 \begin{equation*}
  \begin{split}
      \mu^{\varepsilon,\Lambda} & \ge \left(\frac{a\kappa}{2\pi}-{Wa^2}\right)\log{\frac{1}{\varepsilon}} +\frac{W}{2\kappa}\log\frac{1}{\varepsilon}\int_Dr^2 \zeta^{\varepsilon, \Lambda}  d\nu-|2\delta_0-1|\kappa^{-1}\int_D\partial_sJ(r,\Lambda)\zeta^{\varepsilon,\Lambda}d\nu \\
       &\ \ \ \ \  -(C+o_\varepsilon(\Lambda))\left(1+\left(\int_D \partial_sJ(r,\Lambda)\zeta^{\varepsilon,\Lambda}d\nu\right)^\frac{1}{2}\right) .
  \end{split}
  \end{equation*}
\end{lemma}

It is clear to see that $\partial_sJ(r,\Lambda)\le \partial_sJ(2r_*,\Lambda)$ for $r\le 2r_*$. Hence $\int_D \partial_sJ(r,\Lambda)\zeta^{\varepsilon,\Lambda}d\nu\le C^*$. As a consequence of Lemmas \ref{le4}, \ref{le5} and \ref{le7}, we have
\begin{corollary}
  For each fixed $\Lambda$, there holds $\int_D\zeta^{\varepsilon, \Lambda} d\nu=\kappa$ when $\varepsilon$ is sufficiently small.
\end{corollary}

We now introduce the function $\Gamma$ as follows
\begin{equation*}
  \Gamma(t)=\frac{\kappa t}{2\pi}-Wt^2,\ t\in[0,+\infty).
\end{equation*}
Recall that $r_*=\frac{\kappa}{4\pi W}$. It is easy to check that $\Gamma(r_*)=\max_{t\in[0,+\infty)}\Gamma(t)$.

Next, we show that in order to maximize the energy, those vortices constructed above
must be concentrated as $\varepsilon$ tends to zero. We reach our goal by several steps as follows.

Let $\Theta^{\varepsilon, \Lambda}_-:=\inf\{r\mid(r,0)\in {supp}(\zeta^{\varepsilon, \Lambda})\}$,\ $\Theta^{\varepsilon, \Lambda}_+:=\sup\{r\mid(r,0)\in {supp}(\zeta^{\varepsilon, \Lambda})\}$.
\begin{lemma}\label{le8}
  For each fixed $\Lambda$, $\lim_{\varepsilon\to 0^+}\Theta^{\varepsilon, \Lambda}_-=r_*$.
\end{lemma}

\begin{proof}
 Let $(r_\varepsilon,z_\varepsilon)\in {supp}(\zeta^{\varepsilon, \Lambda})$.  Then we have
 \begin{equation*}
   \mathcal{K}\zeta^{\varepsilon, \Lambda}(r_\varepsilon,z_\varepsilon)-\frac{W(r_\varepsilon)^2}{2}\log{\frac{1}{\varepsilon}}\ge \mu^{\varepsilon, \Lambda}.
 \end{equation*}
 Let $\sigma(r_\varepsilon,z_\varepsilon,r',z')$ be defined by $\eqref{2-1}$ and $\gamma\in(0,1)$.
 By the Green’s representation, we have
\begin{equation*}
\begin{split}
   \mathcal{K}\zeta^{\varepsilon, \Lambda}(r_\varepsilon,z_\varepsilon)&= \int_{D}K(r_\varepsilon,z_\varepsilon,r',z')\zeta^{\varepsilon, \Lambda}(r',z')r'dr'dz'\\
                       &=\left(\int_{D\cap\{\sigma>\varepsilon^\gamma\}}+\int_{D\cap\{\sigma \le \varepsilon^\gamma\}}\right)K(r_\varepsilon,z_\varepsilon,r',z')\zeta^{\varepsilon, \Lambda}(r',z') r'dr'dz'\\
                       &:=I_1+I_2.
\end{split}
\end{equation*}
For the first term $I_1$, we can use $\eqref{Tadiewrong}$ to obtain
\begin{equation}\label{2-14}
\begin{split}
   I_1 &=\int_{D\cap\{\sigma>\varepsilon^\gamma\}}K(r_\varepsilon,z_\varepsilon,r',z')\zeta^{\varepsilon, \Lambda}(r',z') r'dr'dz'\\
       &\le \frac{(r_\varepsilon)^{\frac{1}{2}}}{4\pi} \sinh^{-1}(\frac{1}{\varepsilon^\gamma})\int_{D\cap\{\sigma>\varepsilon^\gamma\}}\zeta^{\varepsilon, \Lambda}(r',z')r'^{{3}/{2}}dr'dz'\\
       &\le \frac{r_*}{2\pi} \sinh^{-1}(\frac{1}{\varepsilon^\gamma})\int_{D\cap\{\sigma>\varepsilon^\gamma\}}\zeta^{\varepsilon, \Lambda}(r',z')r'dr'dz'\\
       &\le \frac{\kappa r_*}{2\pi} \sinh^{-1}(\frac{1}{\varepsilon^\gamma}).
\end{split}
\end{equation}
On the other hand, notice that $D\cap\{\sigma \le \varepsilon^\gamma\} \subseteq B_{4r_*\varepsilon^\gamma}\big((r_\varepsilon,z_\varepsilon)\big)$. In view of Lemma $\ref{le2}$, we get
\begin{equation}\label{2-15}
\begin{split}
   I_2 &=\int_{D\cap\{\sigma \le \varepsilon^\gamma\}}K(r_\varepsilon,z_\varepsilon,r',z')\zeta^{\varepsilon, \Lambda}(r',z')r'dr'dz'\\
       &\le \frac{(r_\varepsilon)^2+C^*\varepsilon^\gamma}{2\pi}\int_{D\cap\{\sigma \le \varepsilon^\gamma\}}\log [(r_\varepsilon-r')^2+(z_\varepsilon-z')^2]^{-\frac{1}{2}}\zeta^{\varepsilon, \Lambda}(r',z')dr'dz'+C^*\\
       &\le \frac{(r_\varepsilon)^2+C^*\varepsilon^\gamma}{2\pi}\log\frac{1}{\varepsilon} \int_{D\cap\{\sigma \le \varepsilon^\gamma\}}\zeta^{\varepsilon, \Lambda}(r',z')dr'dz'+C^*\\
       &\le \frac{(r_\varepsilon)^2}{2\pi}\log\frac{1}{\varepsilon} \int_{D\cap\{\sigma \le \varepsilon^\gamma\}}\zeta^{\varepsilon, \Lambda}(r',z')dr'dz'+C^*\\
       &\le \frac{r_\varepsilon}{2\pi}\log\frac{1}{\varepsilon} \int_{B_{4r_*\varepsilon^\gamma}\big((r_\varepsilon,z_\varepsilon)\big)}\zeta^{\varepsilon, \Lambda} d\nu+C^*,
\end{split}
\end{equation}
where we have used an easy rearrangement inequality. Therefore, by Lemma $\ref{le7}$, we conclude that for any $a\in(r_*/2,2r_*)$
\begin{equation*}
\begin{split}
   \frac{r_\varepsilon}{2\pi}\log\frac{1}{\varepsilon} \int_{B_{4r_*\varepsilon^\gamma}\big((r_\varepsilon,z_\varepsilon)\big)}\zeta^{\varepsilon, \Lambda} d\nu&+\frac{\kappa r_*}{2\pi} \sinh^{-1}(\frac{1}{\varepsilon^\gamma})-\frac{W(r_\varepsilon)^2}{2}\log{\frac{1}{\varepsilon}} \\
     & \ge \left(\frac{a\kappa}{2\pi}-{Wa^2}\right)\log{\frac{1}{\varepsilon}}+\frac{W}{2\kappa}\log{\frac{1}{\varepsilon}}\int_Dr^2\zeta^{\varepsilon, \Lambda} d\nu-C^*.
\end{split}
\end{equation*}
Divide both sides of the above inequality by $\log\frac{1}{\varepsilon}$, we obtain
\begin{equation}\label{2-16}
\begin{split}
   \Gamma(r_\varepsilon)\ge \frac{r_\varepsilon}{2\pi}\int_{B_{4r_*\varepsilon^\gamma}\big((r_\varepsilon,z_\varepsilon)\big)}\zeta^{\varepsilon, \Lambda} d\nu-W(r_\varepsilon)^2 \ge &~ \Gamma(a)+\frac{W}{2\kappa}\left(\int_Dr^2\zeta^{\varepsilon, \Lambda} d\nu-\kappa(r_\varepsilon)^2\right)   \\
     &-\frac{\kappa r_*}{2\pi} \sinh^{-1}(\frac{1}{\varepsilon^\gamma})/\log\frac{1}{\varepsilon}-C^*/\log\frac{1}{\varepsilon}.
\end{split}
\end{equation}
Notice that
\begin{equation*}
  \int_Dr^2\zeta^{\varepsilon, \Lambda}d\nu \ge \kappa(\Theta^{\varepsilon, \Lambda}_-)^2.
\end{equation*}
Taking $(r_\varepsilon,z_\varepsilon)=(\Theta^{\varepsilon, \Lambda}_-,0)$ and letting $\varepsilon$ tend to $0^+$, we deduce from $\eqref{2-16}$ that
\begin{equation}\label{2-17}
  \liminf_{\varepsilon\to 0^+}\Gamma(\Theta^{\varepsilon, \Lambda}_-)\ge \Gamma(a)-\kappa \gamma r_*/(2\pi).
\end{equation}
Hence we get the desired result by letting $a\to r_*$ and $\gamma \to 0$.
\end{proof}

\begin{lemma} \label{le9}
For each fixed $\Lambda$, we have $\int_Dr^2\zeta^{\varepsilon, \Lambda}d\nu\to \kappa r_*^2$ as $\varepsilon \to 0^+$.
\end{lemma}

\begin{proof}
From $\eqref{2-16}$, we know that for any $\gamma\in(0,1)$,
\begin{equation*}
  0\le \liminf_{\varepsilon\to 0^+}\left[\int_Dr^2\zeta^{\varepsilon, \Lambda} d\nu-\kappa(\Theta^{\varepsilon, \Lambda}_-)^2\right]\le \limsup_{\varepsilon\to 0^+}\left[\int_Dr^2\zeta^{\varepsilon, \Lambda} d\nu-\kappa(\Theta^{\varepsilon, \Lambda}_-)^2\right]\le \frac{\kappa^2 r_*\gamma}{\pi W}.
\end{equation*}
Thus, by letting $\gamma\to 0$, we get
\begin{equation*}
  \lim_{\varepsilon\to 0^+}\int_Dr^2\zeta^{\varepsilon, \Lambda}d\nu= \kappa r_*^2.
\end{equation*}
\end{proof}

From Lemmas $\ref{le8}$ and $\ref{le9}$, we immediately get the following result.
\begin{lemma}\label{le10}
For any $\eta>0$, there holds
\begin{equation*}
  \lim_{\varepsilon\to 0^+}\int_{D\cap\{r\ge r_*+\eta\}}\zeta^{\varepsilon, \Lambda} d\nu=0.
\end{equation*}
\end{lemma}

\begin{lemma}\label{le11}
For each fixed $\Lambda$, $\lim_{\varepsilon\to 0^+}\Theta^{\varepsilon, \Lambda}_+=r_*$.
\end{lemma}

\begin{proof}
From $\eqref{2-16}$, we obtain
\begin{equation*}
\begin{split}
   \frac{r_*}{\pi}\liminf_{\varepsilon\to 0^+}\int_{B_{4r_*\varepsilon^\gamma}\big((\Theta^{\varepsilon, \Lambda}_+,0)\big)}\zeta^{\varepsilon,\Lambda} d\nu  & \ge \Gamma(r_*)+\frac{W}{2}\liminf_{\varepsilon\to 0^+}(\Theta^{\varepsilon, \Lambda}_+)^2+\frac{Wr_*^2}{2}-\frac{\kappa r_*\gamma}{2\pi} \\
     & \ge \frac{\kappa r_*}{2\pi}-\frac{\kappa r_*\gamma}{2\pi}.
\end{split}
\end{equation*}
Hence
\begin{equation*}
  \liminf_{\varepsilon\to 0^+}\int_{B_{4r_*\varepsilon^\gamma}\big((\Theta^{\varepsilon, \Lambda}_+,0)\big)}\zeta^{\varepsilon, \Lambda} d\nu\ge\kappa\left(\frac{ 1}{2}-\frac{\gamma}{2}\right).
\end{equation*}
Now the desired result clearly follows from Lemma $\ref{le10}$.
\end{proof}

\begin{lemma}\label{le12}
Let $\Lambda$ be fixed. Then for any number $\gamma\in(0, 1)$, there holds $$diam\big({supp}(\zeta^{\varepsilon,\Lambda})\big) \le 8r_* \varepsilon^{\gamma}$$ provided $\varepsilon$ is small enough. Consequently,
\begin{equation*}
  \begin{split}
     \lim_{\varepsilon\to 0^+}dist_{\mathcal{C}_{r_{_*}}}\left({supp}(\zeta^{\varepsilon,\Lambda})\right) & =0, \\
      dist\left( {supp}(\zeta^{\varepsilon,\Lambda}), \partial D \right) & >0.
  \end{split}
\end{equation*}
\end{lemma}

\begin{proof}
  Let us use the same notation as in the proof of Lemma $\ref{le8}$. Recalling that $\int_{D}\zeta^{\varepsilon,\Lambda} d\nu=\kappa, $ so it suffices to prove that
  \begin{equation*}
  \int_{B_{4r_*\varepsilon^\gamma}\left((r_\varepsilon,z_\varepsilon)\right)}\zeta^{\varepsilon,\Lambda} d\nu>\kappa/2, \ \ \forall\, (r_\varepsilon,z_\varepsilon)\in {supp}(\zeta^{\varepsilon,\Lambda}).
  \end{equation*}
  Firstly, from Lemma $\ref{le11}$ we know that $r_\varepsilon\to r_*$  as $\varepsilon \to 0^+$.
  By $\eqref{2-16}$ we get
\begin{equation}\label{2-17}
  \liminf_{\varepsilon\to 0^+}\int_{B_{4r_*\varepsilon^\gamma}\big((r_\varepsilon,z_\varepsilon)\big)}\zeta^{\varepsilon,\Lambda} d\nu \ge \kappa\left(1-\gamma\right),
\end{equation}
which implies the desired result for $0<\gamma<1/2$  . It follows that $diam\big({supp}(\zeta^{\varepsilon,\Lambda})\big)\le C/\log\frac{1}{\varepsilon}$ provided $\varepsilon$ is small enough. With this in hand, we can improve $\eqref{2-14}$ as follows
\begin{equation*}
\begin{split}
   I_1 &=\int_{D\cap\{\sigma>\varepsilon^\gamma\}}K(r_\varepsilon,z_\varepsilon,r',z')\zeta^{\varepsilon,\Lambda}(r',z') r'dr'dz'\\
       &\le \frac{(r_\varepsilon)^{\frac{1}{2}}}{2\pi} \sinh^{-1}(\frac{1}{\varepsilon^\gamma})\int_{D\cap\{\sigma>\varepsilon^\gamma\}}\zeta^{\varepsilon,\Lambda}(r',z')r'^{\frac{3}{2}}dr'dz'\\
       &\le \frac{\kappa r_\varepsilon}{2\pi} \sinh^{-1}(\frac{1}{\varepsilon^\gamma})+C.
\end{split}
\end{equation*}
We can now repeat the proof and sharpen $\eqref{2-17}$ as follows
\begin{equation*}
  \liminf_{\varepsilon\to 0^+}\int_{B_{4r_*\varepsilon^\gamma}\left((r_\varepsilon,z_\varepsilon)\right)}\zeta^{\varepsilon,\Lambda}d\nu \ge \kappa\left(1-\frac{\gamma}{2}\right)>\kappa/2,
\end{equation*}
which implies the desired result and completes the proof.
\end{proof}

The following lemma shows that  $\psi^{\varepsilon,\Lambda}$ has a priori upper bound with respect to $\Lambda$.
\begin{lemma}\label{le13}
  Let $\Lambda$ be fixed. Then for all sufficiently small $\varepsilon$, we have
  \begin{equation*}
\begin{split}
     \psi^{\varepsilon, \Lambda}(r,z) \le  &\ |2\delta_0-1|\kappa^{-1}\int_D\partial_sJ(r,\Lambda)\zeta^{\varepsilon,\Lambda}d\nu \\
     & +(C+o_\varepsilon(\Lambda))\left(1+\left(\int_D \partial_sJ(r,\Lambda)\zeta^{\varepsilon,\Lambda}d\nu\right)^\frac{1}{2}\right)+C\log \Lambda,\ \ \ \ a.e.\  \text{in}\  D.
\end{split}
\end{equation*}
\end{lemma}

\begin{proof}
  For any $(r,z)\in supp\,(\zeta^{\varepsilon,\Lambda})$, by Lemmas \ref{le1}, \ref{le2}, \ref{le11} and \ref{le12}, we have
\begin{equation*}
  \begin{split}
      \psi^{\varepsilon,\Lambda}(r,z)& =\mathcal{K}\zeta^{\varepsilon,\Lambda}(r,z)-\frac{Wr^2}{2}\log{\frac{1}{\varepsilon}}-\mu^{\varepsilon,\Lambda} \\
       & = \int_D K(r,z,r',z')\zeta^{\varepsilon,\Lambda}r'dr'dz'-\frac{Wr^2}{2}\log{\frac{1}{\varepsilon}}-\mu^{\varepsilon,\Lambda} \\
       & \le \frac{\Theta^{\varepsilon, \Lambda}_++O(\varepsilon^\frac{1}{2})}{2\pi}\int_D\log{\big[(r-r')^2+(z-z')^2 \big]^{-\frac{1}{2}}}\zeta^{\varepsilon,\Lambda}(r', z')r'dr'dz'\\
       & \ \ \ \ \ \ \ \ \ \ \ \ \ \ \ \ \ \ \ \  \ \ -\frac{W(\Theta^{\varepsilon, \Lambda}_-)^2}{2}\log{\frac{1}{\varepsilon}}-\mu^{\varepsilon,\Lambda}+C \\
       & \le \frac{(\Theta^{\varepsilon, \Lambda}_+)^2+O(\varepsilon^\frac{1}{2})}{2\pi}\int_D\log{\big[(r-r')^2+(z-z')^2 \big]^{-\frac{1}{2}}}\zeta^{\varepsilon,\Lambda}(r', z')dr'dz' \\
          & \ \ \ \ \ \ \ \ \ \ \ \ \ \ \ \ \ \ \ \  \ \ -\frac{W(\Theta^{\varepsilon, \Lambda}_-)^2}{2}\log{\frac{1}{\varepsilon}}-\mu^{\varepsilon,\Lambda}+C \\
       & \le \frac{(\Theta^{\varepsilon, \Lambda}_+)^2+O(\varepsilon^\frac{1}{2})}{2\pi}\log\frac{\sqrt{\Lambda}}{\varepsilon}\int_D\zeta^{\varepsilon,\Lambda}(r', z')dr'dz'-\frac{W(\Theta^{\varepsilon, \Lambda}_+)^2}{2}\log{\frac{1}{\varepsilon}}-\mu^{\varepsilon,\Lambda}+C \\
       & \le \left(\frac{\kappa \Theta^{\varepsilon, \Lambda}_+}{2\pi}-\frac{W(\Theta^{\varepsilon, \Lambda}_+)^2}{2}\right)\log{\frac{1}{\varepsilon}}-\mu^{\varepsilon,\Lambda}+C(\log \Lambda+1), \\
  \end{split}
\end{equation*}
where the positive number $C$ does not depend on $\varepsilon$ and $\Lambda$. On the other hand, by Lemma \ref{le6} we have
\begin{equation*}
  \begin{split}
      \mu^{\varepsilon,\Lambda} & \ge \left(\frac{\kappa \Theta^{\varepsilon, \Lambda}_+}{2\pi}-\frac{W(\Theta^{\varepsilon, \Lambda}_+)^2}{2}\right)\log{\frac{1}{\varepsilon}}-|2\delta_0-1|\kappa^{-1}\int_D\partial_sJ(r,\Lambda)\zeta^{\varepsilon,\Lambda}d\nu \\
       &\ \ \ \ \  -(C+o_\varepsilon(\Lambda))\left(1+\left(\int_D \partial_sJ(r,\Lambda)\zeta^{\varepsilon,\Lambda}d\nu\right)^\frac{1}{2}\right) .
  \end{split}
  \end{equation*}
Thus we have
\begin{equation*}
\begin{split}
     \psi^{\varepsilon, \Lambda}(r,z) \le  &\ |2\delta_0-1|\kappa^{-1}\int_D\partial_sJ(r,\Lambda)\zeta^{\varepsilon,\Lambda}d\nu \\
     & +(C+o_\varepsilon(\Lambda))\left(1+\left(\int_D \partial_sJ(r,\Lambda)\zeta^{\varepsilon,\Lambda}d\nu\right)^\frac{1}{2}\right)+C\log \Lambda,
\end{split}
\end{equation*}
 and the proof is thus complete.
\end{proof}

Now we can choose a suitable $\Lambda$ such that the patch part in \eqref{2-6} will vanish when $\varepsilon$ is suffciently small.
\begin{lemma}\label{le14}
   There exists $\Lambda_0>\max\{1,g(0^+)\}$ such that
 \begin{equation*}
   \zeta^{\varepsilon, \Lambda_0}=\frac{1}{ \varepsilon^2}i(r,\psi^{\varepsilon, \Lambda_0}),\ \ \ \ \ a.e.\  \text{in}\  D,
 \end{equation*}
 provided $\varepsilon>0$ is sufficiently small.
\end{lemma}

\begin{proof}
  Notice that on the patch part $\{\zeta^{\varepsilon, \Lambda}={\Lambda\varepsilon^{-2}}\}$, we have
\begin{equation}\label{2-18}
  \psi^{\varepsilon, \Lambda}(r,z)\ge \partial_s J(r,\Lambda)= \partial_s J(r_*,\Lambda)+o_\varepsilon(\Lambda),
\end{equation}
On the other hand, by Lemmas \ref{le12} and \ref{le13}, we have
\begin{equation}\label{2-19}
\begin{split}
   \psi^{\varepsilon,\Lambda}& \le |2\delta_0-1|\kappa^{-1}\int_D\partial_sJ(r,\Lambda)\zeta^{\varepsilon,\Lambda}d\nu\\
    & \ \ \ \ \ \ +(C+o_\varepsilon(\Lambda))\left(1+\left(\int_D \partial_sJ(r,\Lambda)\zeta^{\varepsilon,\Lambda}d\nu\right)^\frac{1}{2}\right)+C\log \Lambda \\
     & \le |2\delta_0-1|\partial_s J(r_*,\Lambda) +(C+o_\varepsilon(\Lambda))\left(1+\left(\partial_sJ(r_*,\Lambda)\right)^\frac{1}{2}\right)+C\log\Lambda      \ \              \ \ \ \ \ a.e.\  \text{in}\  D.
\end{split}
\end{equation}
By assumption $(a4)$, we see that for all $\tau>0$
\begin{equation*}
  \lim_{s\to+\infty}\left(\partial_s J(r_*,s)-\tau\log s \right)=+\infty.
\end{equation*}
Hence we can choose $\Lambda_0>\max\{1,g(0^+)\}$ such that
\begin{equation}\label{2-20}
  (1-|2\delta_0-1|)\partial_s J(r_*,\Lambda_0)-(C+1)\left(1+\left(\partial_sJ(r_*,\Lambda_0)\right)^\frac{1}{2}\right)+C\log\Lambda_0>1.
\end{equation}
From \eqref{2-18},\eqref{2-19} and \eqref{2-20}, we conclude that $m_2\left( \{\zeta^{\varepsilon, \Lambda_0}={\Lambda_0\varepsilon^{-2}}\}\right)=0$ when $\varepsilon$ is sufficiently small, which completes the proof.
\end{proof}

In fact, we have
\begin{lemma}\label{le15}
For all sufficiently small $\varepsilon$, we have
  \begin{equation*}
   \zeta^{\varepsilon, \Lambda_0}=\frac{1}{ \varepsilon^2}i(r,\psi^{\varepsilon, \Lambda_0}),\ \ \ \ \ a.e.\  \text{in}\  \Pi,
 \end{equation*}
\end{lemma}

\begin{proof}
By Lemma \ref{le14}, it suffices to show that $\psi^{\varepsilon, \Lambda_0}=0$ a.e. on $\Pi\backslash D$. By Lemma \ref{le12}, we have
 \begin{equation*}
  \begin{split}
      \mathcal{L}\psi^{\varepsilon, \Lambda_0}&=0\ \ \text{in}\  \Pi\backslash D, \\
      \psi^{\varepsilon, \Lambda_0}&\le 0\ \ \text{on}\  \partial \left( \Pi\backslash D\right), \\
     \psi^{\varepsilon, \Lambda_0} &\le 0 \ \ \text{at}\  \infty .
  \end{split}
\end{equation*}
By the maximum principle, we conclude that $\psi^{\varepsilon, \Lambda_0}=0$ a.e. on $\Pi\backslash D$. The proof is therefore complete.
\end{proof}

In the rest of this section, we shall abbreviate $(\zeta^{\varepsilon, \Lambda_0},\psi^{\varepsilon, \Lambda_0}, \mu^{\varepsilon, \Lambda_0} , \Theta^{\varepsilon, \Lambda}_\pm)$ to $(\zeta^\varepsilon, \psi^\varepsilon, \mu^\varepsilon, \Theta^\varepsilon_\pm)$.

We now sharpen Lemmas \ref{le5} and \ref{le7} as follows.

\begin{lemma}\label{le16}
As $\varepsilon\to 0^+$, we have
\begin{align}
\label{2-21} \mathcal{E}_\varepsilon(\zeta^\varepsilon) & =\left(\frac{\kappa^2r_*}{4\pi}-\frac{\kappa Wr_*^2}{2}\right)\log{\frac{1}{\varepsilon}}+O(1), \\
\label{2-22}  \mu^\varepsilon & =\left(\frac{\kappa r_*}{2\pi}-\frac{Wr_*^2}{2}\right)\log{\frac{1}{\varepsilon}}+O(1).
\end{align}
\end{lemma}

\begin{proof}
We first prove $\eqref{2-21}$. According to Lemma $\ref{le12}$, there holds
\begin{equation*}
  supp\,(\zeta^\varepsilon)\subseteq B_{8r_*\varepsilon^{\frac{1}{2}}}\left((\Theta^{\varepsilon}_+,0)\right)
\end{equation*}
for all sufficiently small $\varepsilon$. Hence, by Lemmas $\ref{le1}$ and $\ref{le2}$, we have
\begin{equation*}
  \int_D\zeta^\varepsilon \mathcal{K} \zeta^\varepsilon d\nu \le\frac{(\Theta^{\varepsilon}_+)^3}{2\pi}\iint_{D\times D}\log[(r-r')^2+(z-z')^2]^{-\frac{1}{2}}\zeta^\varepsilon(r,z)\zeta^\varepsilon(r',z')drdzdr'dz'+O(1).
\end{equation*}
By Lemma 4.2 of \cite{Tur83}, we have
\begin{equation*}
\iint_{D\times D}\log[(r-r')^2+(z-z')^2]^{-\frac{1}{2}}\zeta^\varepsilon(r,z)\zeta^\varepsilon(r',z')drdzdr'dz'\le \log\frac{1}{\varepsilon}\left(\int_D\zeta^\varepsilon drdz\right)^2+O(1).
\end{equation*}
Notice that
\begin{equation*}
 \int_D\zeta^\varepsilon drdz\le \frac{\kappa}{\Theta^\varepsilon_-}.
\end{equation*}
Therefore
\begin{equation*}
   \int_D\zeta^\varepsilon \mathcal{K} \zeta^\varepsilon d\nu\le \frac{\kappa^2(\Theta^\varepsilon_+)^3}{2\pi (\Theta^\varepsilon_-)^2}\log\frac{1}{\varepsilon}+O(1)\le \frac{\kappa^2\Theta^\varepsilon_-}{2\pi}\log\frac{1}{\varepsilon}+O(1),
\end{equation*}
from which we conclude that
\begin{equation*}
  \mathcal{E}_\varepsilon(\zeta^\varepsilon)\le \left(\frac{\kappa^2\Theta^\varepsilon_-}{4\pi}-\frac{\kappa W(\Theta^\varepsilon_-)^2}{2}\right)\log{\frac{1}{\varepsilon}}+O(1)\le \left(\frac{\kappa^2r_*}{4\pi}-\frac{\kappa Wr_*^2}{2}\right)\log{\frac{1}{\varepsilon}}+O(1).
\end{equation*}
Combining the above inequality with Lemma $\ref{le5}$, we get $\eqref{2-21}$. Note that
\begin{equation*}
     2\mathcal{E}_\varepsilon(\zeta^\varepsilon)
       = \int_D{\zeta^\varepsilon }\psi^\varepsilon d\nu-{\frac{W}{2}\log{\frac{1}{\varepsilon}}}\int_{D}r^2 \zeta^\varepsilon d\nu -\frac{2}{\varepsilon^2}\int_D J(r,\varepsilon^2\zeta^\varepsilon)d\nu+\kappa \mu^\varepsilon ,
\end{equation*}
from which $\eqref{2-22}$ follows and the proof is thus complete.
\end{proof}

 We can also improve Lemma \ref{le12} as follows.
\begin{lemma}\label{le17}
There exists constants $R_0, R_1>0$ independent of $\varepsilon$ such that $$R_0\varepsilon\le diam\left(supp(\zeta^\varepsilon)\right)\le R_1\varepsilon.$$
\end{lemma}

\begin{proof}
First, we clearly have
\begin{equation*}
  \kappa=\int_D\zeta^\varepsilon d\nu\le C [diam\left(supp(\zeta^\varepsilon)\right)]^2/\varepsilon^2,
\end{equation*}
from which it follows that there exists a $R_0>0$ such that $$ diam\left(supp(\zeta^\varepsilon)\right)\ge R_0\varepsilon.$$
So, it remains to be proved that $diam\left(supp(\zeta^\varepsilon)\right)=O(\varepsilon)$.
Let us use the same notation as in the proof of Lemma $\ref{le6}$. According to Lemma $\ref{le12}$, there holds
\begin{equation*}
  supp\,(\zeta^\varepsilon)\subseteq B_{8r_*\varepsilon^{\frac{1}{2}}}\left((\Theta^{\varepsilon}_+,0)\right)
\end{equation*}
for all sufficiently small $\varepsilon$.
Let $R>1$ to be determined. By Lemma $\ref{le2}$, we have
\begin{equation}\label{2-23}
\begin{split}
   I_1 &=\int_{D\cap\{\sigma>R\varepsilon\}}K(r_\varepsilon,z_\varepsilon\,r',z')\zeta^\varepsilon(r',z')r'dr'dz'\\
       &\le \frac{(\Theta^\varepsilon_-)^2+O(\varepsilon^{\frac{1}{2}})}{2\pi}\int_{D\cap\{\sigma>R\varepsilon\}}\log{\frac{1}{\sigma}}\zeta^\varepsilon(r',z') dr'dz'+C\\
       &\le \frac{(\Theta^\varepsilon_-)^2+O(\varepsilon^{\frac{1}{2}})}{2\pi}\log{\frac{1}{R\varepsilon}}\int_{D\cap\{\sigma>\varepsilon\}}\zeta^\varepsilon(r',z') dr'dz'+C\\
       &\le \frac{\Theta^\varepsilon_-}{2\pi}\log{\frac{1}{R\varepsilon}}\int_{D\cap\{\sigma>R\varepsilon\}}\zeta^\varepsilon d\nu+C,
\end{split}
\end{equation}
and
\begin{equation}\label{2-24}
\begin{split}
   I_2 &=\int_{D\cap\{\sigma \le R\varepsilon\}}K(r_\varepsilon,z_\varepsilon,r',z')\zeta^\varepsilon(r',z')r'dr'dz'\\
       &\le \frac{(\Theta^\varepsilon_-)^2+O(\varepsilon^\frac{1}{2})}{2\pi}\int_{D\cap\{\sigma \le R\varepsilon\}}\log\frac{1}{[(r_\varepsilon-r')^2+(z_\varepsilon-z')^2]^{\frac{1}{2}}}\zeta^\varepsilon(r',z')dr'dz'+C\\
       &\le \frac{(\Theta^\varepsilon_-)^2+O(\varepsilon^\frac{1}{2})}{2\pi}\log\frac{1}{\varepsilon}\int_{D\cap\{\sigma \le R\varepsilon\}}\zeta^\varepsilon dr'dz'+C\\
       & \le \frac{\Theta^\varepsilon_-}{2\pi}\log\frac{1}{\varepsilon}\int_{D\cap\{\sigma \le R\varepsilon\}}\zeta^\varepsilon d\nu+C.
\end{split}
\end{equation}
On the other hand, by Lemma $\ref{le16}$, one has
\begin{equation}\label{2-25}
  \mu^\varepsilon\ge\left(\frac{\kappa \Theta^\varepsilon_-}{2\pi}-\frac{W(\Theta^\varepsilon_-)^2}{2}\right)\log\frac{1}{\varepsilon}-C.
\end{equation}
Combining $\eqref{2-23}$, $\eqref{2-24}$ and $\eqref{2-25}$, we obtain
\begin{equation*}
\begin{split}
     \frac{\kappa \Theta^\varepsilon_-}{2\pi}\log\frac{1}{\varepsilon} & \le \frac{\Theta^\varepsilon_-}{2\pi}\log{\frac{1}{R\varepsilon}}\int_{D\cap\{\sigma>R\varepsilon\}}\zeta^\varepsilon d\nu+\frac{\Theta^\varepsilon_-}{2\pi}\log\frac{1}{\varepsilon}\int_{D\cap\{\sigma \le R\varepsilon\}}\zeta^\varepsilon d\nu+C\\
     & \le \frac{\kappa \Theta^\varepsilon_-}{2\pi}\log\frac{1}{\varepsilon}+\frac{\Theta^\varepsilon_-}{2\pi}\log\frac{1}{R}\int_{D\cap\{\sigma>R\varepsilon\}}\zeta^\varepsilon d\nu+C.
\end{split}
\end{equation*}
Hence
\begin{equation*}
  \int_{D\cap\{\sigma \le R\varepsilon\}}\zeta^\varepsilon d\nu\ge \kappa-\frac{C}{\log R}.
\end{equation*}
Taking $R$ large enough such that $C(\log R)^{-1}<\kappa/2$, we obtain
\begin{equation*}
  \int_{D\cap{B_{4r_*R\varepsilon}\left((r_\varepsilon,z_\varepsilon)\right)}}\zeta^\varepsilon d\nu\ge\int_{D\cap\{\sigma \le R\varepsilon\}}\zeta^\varepsilon d\nu> \frac{\kappa}{2}.
\end{equation*}
Taking $R_1=8r_*R$, we clearly get the desired result.
\end{proof}

We now investigate the asymptotic shape of the optimal vortices.
Define the center of vorticity by
\begin{equation*}
  X^\varepsilon=(R^\varepsilon,0)=\frac{\int_D x\zeta^\varepsilon(x)d\textit{m}_2(x)}{\int_D \zeta^\varepsilon(x)d\textit{m}_2(x)}.
\end{equation*}
Then by Lemma $\ref{le12}$ we know that $X^\varepsilon \to (r_*,0)$ as $\varepsilon \to 0^+$. Let $\phi^\varepsilon $ be defined by
\begin{equation*}
  \phi^\varepsilon(x)={\varepsilon^2}\zeta^\varepsilon(X^\varepsilon+\varepsilon x),\ \ x\in D^\varepsilon:=\{x\in \mathbb{R}^2\mid X^\varepsilon+\varepsilon x \in \Pi\}
\end{equation*}
For convenience, we set $\phi^\varepsilon(x)=0$ if $x\in \mathbb{R}^2 \backslash D^\varepsilon$. It is easy to see that $supp(\phi^\varepsilon)\subset B_{R_0}(0)$ and $0\le \phi^\varepsilon\le \Lambda_0$. Moreover, as $\varepsilon\to 0^+$,
\begin{equation}\label{ad}
  \int_{\mathbb{R}^2} \phi^\varepsilon(x) dm_2(x)=\int_\Pi\zeta^\varepsilon(r,z)drdz=\kappa/r_*+o(1)=4\pi W+o(1).
\end{equation}

 We denote by $\bar{\phi}^\varepsilon$ the symmetric radially nonincreasing Lebesgue-rearrangement of $\phi^\varepsilon$ centered at the origin.

The following result determines the asymptotic nature of $\zeta^\varepsilon$ in terms of its scaled version.

\begin{lemma}\label{le18}
Every accumulation point of the family $\{\phi^\varepsilon:\varepsilon>0\}$ as $\varepsilon \to 0^+$, in the weak topology of $L^2$, must be a radially nonincreasing function.
\end{lemma}

\begin{proof}
Let us assume that $\phi^\varepsilon\to \phi$ and $\bar{\phi}^\varepsilon \to h$ weakly in $L^2$ for some functions $\phi$ and $h$ as $\varepsilon \to 0^+$. Firstly, by virtue of Riesz' rearrangement inequality, we have
\begin{equation*}
     \iint\limits_{\mathbb{R}^2\times\mathbb{R}^2}\log\frac{1}{|x-x'|}\phi^\varepsilon(x)\phi^\varepsilon(x')d(x,x')
      \le \iint\limits_{\mathbb{R}^2\times\mathbb{R}^2}\log\frac{1}{|x-x'|}\bar{\phi}^\varepsilon(x)\bar{\phi}^\varepsilon(x')d(x,x').
\end{equation*}
Hence
\begin{equation}\label{2-26}
  \iint\limits_{\mathbb{R}^2\times\mathbb{R}^2}\log\frac{1}{|x-x'|}\phi(x)\phi(x')d(x,x')\le \iint\limits_{\mathbb{R}^2\times\mathbb{R}^2}\log\frac{1}{|x-x'|}h(x)h(x')d(x,x').
\end{equation}
Let $\tilde{\zeta}^\varepsilon$ be defined by $\tilde{\zeta}^\varepsilon(x)=\varepsilon^{-2}\phi^\varepsilon(\varepsilon^{-1}\left(x-X_\varepsilon)\right)$. A direct calculation then deduces as $\varepsilon\to 0^+$,
\begin{equation*}
\begin{split}
     \mathcal{E}_\varepsilon(\zeta^\varepsilon) =&\frac{(\Theta^\varepsilon_-)^3}{4\pi}\iint\limits_{\mathbb{R}^2\times\mathbb{R}^2}\log\frac{1}{|x-x'|}\phi^\varepsilon(x)\phi^\varepsilon(x')d(x,x') \\
     & +\frac{(\Theta^\varepsilon_-)^3}{4\pi}\left(\int_D \zeta^\varepsilon d\textit{m}_2\right)^2\log\frac{1}{\varepsilon}-\frac{W(\Theta^\varepsilon_-)^3}{2}\log\frac{1}{\varepsilon}\int_D\zeta_{_\varepsilon} d\textit{m}_2+\mathcal{R}^\varepsilon_1,
\end{split}
\end{equation*}
and
\begin{equation*}
\begin{split}
     \mathcal{E}_\varepsilon(\tilde{\zeta}^\varepsilon) =&\frac{(\Theta^\varepsilon_-)^3}{4\pi}\iint\limits_{\mathbb{R}^2\times\mathbb{R}^2}\log\frac{1}{|x-x'|}\bar{\phi}^\varepsilon(x)\bar{\phi}^\varepsilon(x')d(x,x') \\
     & +\frac{(\Theta^\varepsilon_-)^3}{4\pi}\left(\int_D \tilde{\zeta}^\varepsilon d\textit{m}_2\right)^2\log\frac{1}{\varepsilon}-\frac{W(\Theta^\varepsilon_-)^3}{2}\log\frac{1}{\varepsilon}\int_D\tilde{\zeta}^\varepsilon d\textit{m}_2+\mathcal{R}^\varepsilon_2,
\end{split}
\end{equation*}
where
\begin{equation*}
  \lim_{\varepsilon\to 0^+}\left(\mathcal{R}^\varepsilon_1-\mathcal{R}^\varepsilon_2\right)=0.
\end{equation*}
Notice that
\begin{equation*}
  0\le \tilde{\zeta}^\varepsilon \le \frac{1}{\varepsilon^2},\ \ \ supp(\tilde{\zeta}^\varepsilon)\subset D, \ \ \  \int_D \tilde{\zeta}^\varepsilon d\nu=\kappa+O(\varepsilon).
\end{equation*}
It is not hard to check that $ \mathcal{E}_\varepsilon(\tilde{\zeta}^\varepsilon) \le   \mathcal{E}_\varepsilon(\zeta^\varepsilon) +o(1)$ as $\varepsilon\to 0^+$. Thus, we have
\begin{equation*}
  \iint\limits_{\mathbb{R}^2\times\mathbb{R}^2}\log\frac{1}{|x-x'|}\phi(x)\phi(x')d(x,x')\le \iint\limits_{\mathbb{R}^2\times\mathbb{R}^2}\log\frac{1}{|x-x'|}h(x)h(x')d(x,x'),
\end{equation*}
which together with $\eqref{2-26}$ deduces
\begin{equation*}
    \iint\limits_{\mathbb{R}^2\times\mathbb{R}^2}\log\frac{1}{|x-x'|}\phi(x)\phi(x')d(x,x')= \iint\limits_{\mathbb{R}^2\times\mathbb{R}^2}\log\frac{1}{|x-x'|}h(x)h(x')d(x,x').
\end{equation*}
By Lemma 3.2 of \cite{BG}, we know that there exists a translation $T$ in $\mathbb{R}^2$ such that $T\phi=h$. Note that
\begin{equation*}
  \int_{\mathbb{R}^2}x\phi(x)d\textit{m}_2=\int_{\mathbb{R}^2}xh(x)d\textit{m}_2=0.
\end{equation*}
Thus $\phi=h$, and the proof is complete.
\end{proof}

Next we study the limiting behavior of the corresponding stream functions $\psi^\varepsilon$.
Define the scaled versions of $\psi^\varepsilon$ as follows
\begin{equation*}
    \Psi^\varepsilon(x):=\psi^\varepsilon(X^\varepsilon+\varepsilon x),\ \ x\in D^\varepsilon.
\end{equation*}

\begin{lemma}\label{le18}
  Up to a subsequence, $\Psi^\varepsilon \to \Psi$ in $C^{1,\gamma}_{\text{loc}}(\mathbb{R}^2)$ for some $\gamma\in(0,1)$, where $\Psi$ is a radial strictly decreasing function in $\mathbb{R}^2$.
\end{lemma}

\begin{proof}
  Recall that
  \begin{equation*}
    \mathcal{L}\psi^\varepsilon=-\frac{1}{r}\frac{\partial}{\partial r}\Big(\frac{1}{r}\frac{\partial\psi^\varepsilon}{\partial r}\Big)-\frac{1}{r^2}\frac{\partial^2\psi^\varepsilon}{\partial z^2}=\zeta^\varepsilon.
  \end{equation*}
  A simple calculation yields
  \begin{equation}\label{2-27}
  -\frac{\partial^2}{\partial x_1^2}\Psi^\varepsilon(x)+\frac{\varepsilon}{R^\varepsilon+\varepsilon x_1}\frac{\partial}{\partial x_1}\Psi^\varepsilon(x)-\frac{\partial^2}{\partial x_2^2}\Psi^\varepsilon(x)=(R^\varepsilon+\varepsilon x_1)^2\phi^\varepsilon(x).
  \end{equation}
  Note that $\{\phi^\varepsilon\}$ is bounded in $L^\infty(\mathbb{R}^2)$. Thus, by classical elliptic estimates, the sequence $\{\Psi^\varepsilon\}$ is bounded in $W^{2,p}_{\text{loc}}(\mathbb{R}^2)$ for every $1\le p<+\infty$. By the Sobolev embedding theorem, we may conclude that $\{\Psi^\varepsilon\}$ is compact in $C^{1,\alpha}_{\text{loc}}(\mathbb{R}^2)$ for every $0<\alpha<1$. Up to a subsequence we may assume $\phi^\varepsilon\to \phi$ weakly-star in $L^\infty(\mathbb{R}^2)$ and $\Psi^\varepsilon \to \Psi$ in $C_{\text{loc}}^{1,\alpha}(\mathbb{R}^2)$. By virtue of \eqref{ad} and \eqref{2-27}, we get
\begin{equation*}
  -\Delta \Psi =r_*\phi \ \ \text{in}\ \mathbb{R}^2,\ \ \  \int_{\mathbb{R}^2}r_*\phi dm_2(x)=\kappa.
\end{equation*}
  Let
  \begin{equation*}
    \tilde{\Psi}:=\frac{1}{2\pi}\int_{\mathbb{R}^2}\log\frac{1}{|x-y|}r_*\phi(y)dm_2(y).
  \end{equation*}
  Hence, we have
  \begin{equation*}
    -\Delta (\Psi-\tilde{\Psi}) =0 \ \ \text{in}\ \mathbb{R}^2.
  \end{equation*}
  Noting that $\Psi$ and $\tilde{\Psi}$ are both bounded from above, by Liouville's theorem, we must have $\Psi\equiv\tilde{\Psi}+C$ for some constant $C$. Obviously $\tilde{\Psi}$ is a radial strictly decreasing function in $\mathbb{R}^2$, and the proof is thus complete.
\end{proof}

Finally, we prove that the geometry of vortex rings constructed above is indeed a topological torus.
\begin{lemma}\label{le19}
  For $\varepsilon$ sufficiently small, $supp(\zeta^\varepsilon)$ is a topological disc.
\end{lemma}

\begin{proof}
  Note that
  \begin{equation*}
    \{x\in\Pi\mid\psi^\varepsilon(x)=0\}=X^\varepsilon+\varepsilon\{y\in \mathbb{R}^2\mid \Psi^\varepsilon(y)=0\}.
  \end{equation*}
  The desired result is thus a immediate consequence of Lemma \ref{le18} by the implicit function theorem.
\end{proof}

Now we are ready to give the proof of Theorem \ref{thm1}.
\begin{proof}[Proof of Theorem \ref{thm1}]
Note that the operator $\mathcal{L}$ has a close connection with the Laplacian of cylindrically symmetric functions on $\mathbb{R}^5$; see, e.g., \cite{AF, BB, B1}. It is standard to find that as $r\to 0^+$,
\begin{equation*}
  \psi^\varepsilon=O(r^2),\ \ \frac{1}{r}\frac{\partial \psi^\varepsilon}{\partial z}=O(r)\ \ \text{and}~ \ \frac{1}{r}\frac{\partial\psi^\varepsilon
     }{\partial r}\  \text{approaches a finite limit.}
\end{equation*}
Now Theorem \ref{thm1} follows from the above lemmas by letting $(g, f)=(-B', HH')$.
\end{proof}

\section{Appendix}
In this appendix, we will give a variant of bathtub principle (see Lieb-Loss \cite{Lieb}, \S1.14), which is used in the proof of Lemma \ref{le3}.

\begin{lemma}[variant of bathtub principle]\label{bathtub}
	Let $ (\Omega, \mathcal{M}, \vartheta) $ be a measure space and let $h$ be a real-valued measurable function on $ \Omega $ such that $\vartheta(\{x\in \Omega\mid h(x)>t\})$ is finite for all $t>0$. Let the number $\varrho\in (0,\vartheta(\Omega))$ be given and define a class of measurable functions on $ \Omega $ by
	\begin{equation*}
	\mathcal{C}=\{\omega\mid 0\leq \omega(x)\leq 1\ \   a.e. \ x\in \Omega,\ \ \ \int_{\Omega}\omega(x)\vartheta(dx)\le \varrho\}.
	\end{equation*}
	Consider the maximization problem
	\begin{equation*}
	\mathcal{I}=\sup_{\omega\in \mathcal{C}}\int_{\Omega}h(x)\omega(x)\vartheta(dx).
	\end{equation*}
Suppose $\mathcal{I}<+\infty$, then $\mathcal{I}$ is attained. Moreover, any maximizer $\varpi$ satisfies
\begin{equation}\label{app1}
  \varpi(x)=1\ \ a.e.\  \text{on}\ \{h>\max\{0,\varsigma\}\}\ \ and\ \ \varpi(x)=0\ \  a.e.\  \text{on}\ \{h<\max\{0,\varsigma\}\},
\end{equation}
where
\begin{equation}\label{app2}
	\varsigma=\inf\{t\mid \vartheta(\{x\in \Omega\mid h(x)>t\})\leq \varrho\}\in \mathbb{R}.
	\end{equation}
If $\varsigma>0$, then $\int_{\Omega}\varpi(x)\vartheta(dx)=\varrho$.
\end{lemma}

This result is classical. We remark that the freedom of the maximizer occurs only on the level set $\{h=\max\{0,\varsigma\}\}$. We refer the reader to \cite{B2,Dou} for some more general discussions.  For the convenience of the reader, we provide its proof here.

\begin{proof}
  Since $\mathcal{I}<+\infty$, there exists some $t_0>0$ such that $\vartheta(\{x\in \Omega\mid h(x)>t_0\})\le \varrho $. It follows that \eqref{app2} is well-defined and $\varsigma\in \mathbb{R}$.

 Case 1: $\varsigma\le 0$. For every $\omega\in \mathcal{C}$, we have
 \begin{equation}\label{app3}
   \int_\Omega h(x)\omega(x)\vartheta(dx)\le \int_{\{h>0\}}h(x)\omega(x)\vartheta(dx)\le \int_{\{h>0\}}h(x)\vartheta(dx),
 \end{equation}
 which implies
 $$\mathcal{I}\le \int_{\{h>0\}}h(x)\vartheta(dx).$$
 Note that $\sigma(\{h>0\})\le \varrho$. It follows that $\chi_{_{\{h>0\}}}\in \mathcal{C}$ and hence $\mathcal{I}=\int_{\{h>0\}}h(x)\vartheta(dx)$. This shows that $\mathcal{I}$ is attained. Next, we prove \eqref{app1}. Indeed, if $\varpi$ is a maximizer, then all inequalities in \eqref{app3} have to be converted to equalities, which clearly implies \eqref{app1}.

 Case 2: $\varsigma>0$.
  For every $\omega\in \mathcal{C}$, we have
  \begin{equation}\label{app4}
  \begin{split}
     \int_\Omega h(x)\omega(x)\vartheta(dx) & \le \int_{\{h>\varsigma\}} h(x)\omega(x)\vartheta(dx)+\int_{\{h=\varsigma\}} h(x)\omega(x)\vartheta(dx)+\int_{\{h<\varsigma\}} h(x)\omega(x)\vartheta(dx) \\
       & \le \int_{\{h>\varsigma\}} h(x)\omega(x)\vartheta(dx)+\varsigma\int_{\{h\le\varsigma\}} \omega(x)\vartheta(dx) \\
       & \le \int_{\{h>\varsigma\}} (h(x)-\varsigma)\omega(x)\vartheta(dx)+\varsigma \varrho \\
       & \le \int_{\{h>\varsigma\}} (h(x)-\varsigma)\vartheta(dx)+\varsigma \varrho,
  \end{split}
  \end{equation}
which implies
  \begin{equation*}
    \mathcal{I}\le \int_{\{h>\varsigma\}} (h(x)-\varsigma)\vartheta(dx)+\varsigma \varrho.
  \end{equation*}
  On the other hand, in this case, it holds
 \begin{equation*}
    \vartheta(\{h>\varsigma\})\le \varrho \le \vartheta(\{ h\ge \varsigma\}).
  \end{equation*}
  Therefore, we can find a $\bar{\omega}\in \mathcal{C}$ with
    \begin{equation*}
    \bar{\omega}(x)=\chi_{_{\{h>s\}}}(x)+\chi_{_{A}}(x)
  \end{equation*}
  for some measurable set $A\subset \{h=\varsigma\}$ satisfying $\sigma(A)=\varrho-\sigma(\{h>\varsigma\})$. Notice that
  \begin{equation*}
    \int_\Omega h(x)\bar{\omega}(x)\vartheta(dx)=\int_{\{h>\varsigma\}} (h(x)-\varsigma)\vartheta(dx)+\varsigma \varrho.
  \end{equation*}
  Hence
  \begin{equation}\label{app5}
    \mathcal{I}=\int_{\{h>\varsigma\}} (h(x)-\varsigma)\vartheta(dx)+\varsigma \varrho.
  \end{equation}
  This shows that $\mathcal{I}$ is attained. If $\varpi$ is a maximizer, then all inequalities in \eqref{app4} have to be converted to equalities, from which we get \eqref{app1}. Finally, suppose $\varsigma>0$, then it follows from \eqref{app1} and \eqref{app5} that $\int_{\Omega}\varpi(x)\vartheta(dx)=\varrho$, and the proof is thus complete.
\end{proof}

{\bf Acknowledgements:} This work was supported by NNSF of China Grant 11831009 and Chinese Academy of Sciences (No. QYZDJ-SSW-SYS021).

\phantom{s}
 \thispagestyle{empty}

\end{document}